\documentclass[letterpaper,11pt]{article}

\usepackage[utf8]{inputenc}
\usepackage[english]{babel}

\usepackage{latexsym}
\usepackage{amssymb}
\usepackage{amsthm}
\usepackage{mathrsfs}
\usepackage{amstext} 
\usepackage{graphicx}
\usepackage{subcaption}
\usepackage{amsfonts}
\usepackage{booktabs}
\usepackage{wrapfig}
\usepackage{sidecap} 
\usepackage{latexsym}
\usepackage{pdfpages}
\usepackage{epic}
\usepackage{fullpage}
\usepackage{color}
\usepackage{curves}
\usepackage{setspace}
\usepackage{amsmath}
\numberwithin{equation}{section}
\usepackage{anysize}
\usepackage{rotating}
\usepackage{enumerate} 
\usepackage{hyperref}
\usepackage{bbm}

\usepackage{graphicx}

\usepackage{xcolor}
\usepackage{tikz}
\usetikzlibrary{arrows.meta,positioning,shapes.geometric,shapes.misc}

\newtheorem{theorem}{Theorem}[section] 
\newtheorem{definition}[theorem]{Definition}

\newtheorem{proposition}[theorem]{Proposition}
\newtheorem{lemma}[theorem]{Lemma}
\newtheorem{remark}[theorem]{Remark}
\newtheorem{algorithm}[theorem]{Algorithm}
\def\R{{\mathbb R}}

\renewcommand{\leq}{\leqslant}
\renewcommand{\geq}{\geqslant}
\numberwithin{equation}{section}

\title{A null controllability data assimilation for the bulk--surface heat equation with dynamic boundary conditions}
\date{}

\author{
	Javier Ram\'irez-Ganga\thanks{Centro de Modelamiento Matem\'atico (CMM) IRL 2807 CNRS-UChile and Departamento de Ingenier\'ia Matem\'atica (DIM), Universidad de Chile, Beauchef 851, Casilla 170-3, Correo 3, Santiago, Chile, e-mail: jramirez@dim.uchile.cl.}
	}

\begin{document}
\maketitle
\begin{abstract}
We address the inverse problem of reconstructing the state at a final time $T_0$ of a parabolic equation with dynamic boundary conditions of Wentzell type, from a distributed interior observation on a subdomain $\omega \subset \Omega$. Our approach combines (i) a Carleman-type observability inequality for the bulk--surface system, with (ii) a Tikhonov-regularized optimal-control reformulation of the data assimilation problem, and (iii) a fully implementable spatio-temporal discretization based on a Lie splitting
that decouples the bulk and surface dynamics. We prove an exact reconstruction theorem for the continuous problem, derive a stability inequality, and analyze a penalized discrete scheme. We then propose an adaptive post-processing step that optimally combines the regularized reconstruction with the observation itself; under noisy data, the resulting estimator automatically selects between full post-processing and the raw reconstruction without prior knowledge of the noise level. Extensive numerical experiments on a two-dimensional Wentzell heat equation validate the method. A semilinear extension to the Allen--Cahn nonlinearity is treated via a Picard outer loop with Schauder fixed-point convergence; two complementary experiments reveal an intrinsic observability--nonlinearity trade-off in the bulk--surface setting.
\end{abstract}


\medskip

\noindent {\bf MSC Classifications (2020):} 35R30, 35K20, 65M60, 93B05.

\section{Introduction}

The reconstruction of the state of a parabolic equation from a distributed observation has been studied intensively over the past two decades, both for applications in geophysics and as a benchmark of inverse-problem techniques \cite{talagrand1987variational,blum1998variational,MR1642564,MR2491591,MR2804643,MR3814604}. For models posed in a bulk--surface setting---where a parabolic equation in a domain $\Omega \subset \R^d$ is coupled with a dynamic boundary equation on $\Gamma = \partial\Omega$ (so-called Wentzell boundary conditions, \cite{MR2215623,MR1890879})---a quantitative observability theory has only become available more recently, in particular through the global Carleman estimate of Maniar, Meyries and Schnaubelt \cite{MR3669656}. The application of this estimate to data assimilation problems remains, to the best of our knowledge, open.

Bulk--surface parabolic equations with dynamic (Wentzell-type) boundary  conditions arise in several distinct applied contexts. In geophysics and  oceanography, they model bulk- surface heat exchange between the ocean  interior and the air--sea interface; for instance, the linearized  quasi-geostrophic model of Garc\'ia, Osses and Puel \cite{MR2804643} is  exactly of this type, and the present framework would extend their data assimilation method to settings where the surface evolution is not a passive  trace but is governed by its own PDE. Furthermore, in cell and membrane biology, these equations describe phase-field and reaction--diffusion  models on cell membranes coupled with bulk cytoplasmic dynamics  \cite{MR3047936}, particularly in scenarios where observation is naturally  intracellular (the bulk domain) while the surface state (such as the membrane concentration of a signalling molecule) is inaccessible and must be reconstructed. They also appear in materials science, specifically in phase separation and coarsening within thin films with surface confinement. In this context, Allen--Cahn type surface nonlinearities encode wetting transitions, and the semilinear extension developed in Section \ref{sec:semilinear} is directly motivated by this application. Finally, in heat transfer engineering, this framework applies to conduction in solids with convective or radiative surface boundary conditions in the dynamic regime, where the surface heat capacity is non-negligible \cite{MR2215623}.

In this paper we propose, analyze, and implement a method that combines the Maniar--Meyries--Schnaubelt Carleman estimate with the Tikhonov-regularized control reformulation pioneered by Puel \cite{MR2491591}. The main objectives of this work are: \begin{enumerate}
\item A \emph{representation formula} (Theorem~\ref{thm:repr}) which expresses the value of the bulk state at $T_0$ in terms of the observation, of the source, and of a ``null control'' function $g^*(\Phi_0)$ that parametrizes a backward adjoint problem with prescribed terminal datum $\Phi_0$.

\item A \emph{penalized scheme} (Theorem~\ref{thm:penalized}) which selects the control as the minimizer of an $\alpha$-regularized cost functional, well-posed at the discrete level after a Galerkin reduction.

\item A \emph{Lie-splitting bulk--surface discretization}, providing an efficient and structure-preserving way to assemble the discrete controllability Gramian $\boldsymbol{\mathcal C}_h$ underlying the reconstruction.

\item An \emph{adaptive post-processing step} that automatically selects, from the data, whether to trust the observation inside $\omega$ or the regularized reconstruction outside; the rule requires no prior calibration of the noise level.

\item An extensive \emph{numerical study} revealing a non-monotone error/radius profile, identifying the structural role of multi-disk observatories, and demonstrating that the adaptive scheme stays near a $\sim 9\,\%$ raw ceiling for moderate noise and below $\sim 16\,\%$ even at $100\,\%$ noise.
\end{enumerate}

The exact null-controllability approach to data assimilation goes back to Puel \cite{MR2491591}: rather than minimizing the mismatch between the observation and a forward simulation (the variational 4D-Var problem, classically used in atmospheric and oceanic data assimilation \cite{le1986variational,courtier1987variational,daley1993atmospheric} but ill-posed when the initial data is unknown), one parametrizes the unknown final state by a backward control and uses the Carleman estimate to obtain a stability bound. Garc\'ia, Osses and Puel \cite{MR2804643} applied the method to the quasi-geostrophic equations of oceanography, where the regularization $\alpha \to 0$ is shown to converge to the exact reconstruction limit. The underlying Carleman machinery for parabolic equations was systematically developed in \cite{MR1406566, MR1987865, MR1362555, MR3041662} and extended to nonlinear and discontinuous settings in \cite{MR1932966,MR2103189}. Our contribution adapts this circle of ideas to the bulk--surface case with Wentzell boundary conditions, building on the Carleman estimate of Maniar, Meyries and Schnaubelt \cite{MR3669656}.

On the discretization side, finite-element methods for bulk--surface problems are by now well understood \cite{MR3047936, MR2050138,MR2743037}. Lie-type splitting methods specifically for dynamic boundary conditions were studied recently by Altmann, Kov\'acs and Zimmer \cite{MR4568436}, who proved convergence of order $\tau$ in the energy norm; the geometric integration framework of \cite{MR2221614} provides further context for splitting schemes. Our implementation follows \cite{MR4568436} but adds the assembly of the discrete controllability Gramian $\boldsymbol{\mathcal C}_h$ as a Gram matrix, symmetric positive semi-definite by construction; see Section \ref{sec:discretization}. From the inverse-problems perspective, our Tikhonov penalization fits the classical regularization theory developed in \cite{MR1408680,MR2459012,MR766231}.

\section{Problem formulation and main results} \label{sec:problem}

\subsection{The continuous problem and functional setting}\label{ssec:cont}

Let $\Omega \subset \R^d$, $d \in \{2,3\}$, be a bounded domain with smooth boundary $\Gamma = \partial \Omega$, and let $T_0 > 0$. We consider the linear bulk--surface heat equation
\begin{equation}
\label{eq:state}
\begin{cases}
\partial_t u - \Delta u = f & \text{in } \Omega \times (0, T_0), \\
\beta\, \partial_t p + \partial_n u - \Delta_{\Gamma} p = f_\Gamma
& \text{on } \Gamma \times (0, T_0), \\
u|_\Gamma = p & \text{on } \Gamma \times (0, T_0),
\end{cases}
\end{equation}
where $\Delta_\Gamma$ is the Laplace--Beltrami operator on $\Gamma$, $\partial_n$ is the outward normal derivative, and $\beta > 0$ is a fixed surface-diffusion parameter. The trace coupling $u|_\Gamma = p$ is enforced strongly.

The natural energy space for \eqref{eq:state} is
\begin{equation}
\label{eq:Hspace}
\mathcal{H} := \bigl\{ (v, q) \in L^2(\Omega) \times L^2(\Gamma) :
v|_\Gamma = q \bigr\},
\end{equation}
endowed with the $\beta$-weighted inner product
\begin{equation}
\label{eq:Hinner}
\bigl( (u, p), (v, q) \bigr)_{\mathcal{H}}
:= \int_\Omega u v\, dx + \beta \int_\Gamma p q\, dS.
\end{equation}
We use the abbreviation $U := (u, p) \in \mathcal{H}$ when discussing the coupled state; the source is similarly written $F := (f, f_\Gamma)$. Well-posedness of \eqref{eq:state} in $C([0,T_0]; \mathcal{H})$ has been established under mild assumptions (\cite{MR3669656,MR4568436}).

\paragraph{The data assimilation problem:} We assume the initial datum $U_0 = (u_0, p_0) \in \mathcal{H}$ is \emph{unknown}, and that the bulk state is measured on a subdomain $\omega \subseteq \Omega$:
\begin{equation}
\label{eq:obs}
h(t,x) := u(t,x)|_{\omega \times (0,T_0)},
\qquad h \in L^2(0,T_0; L^2(\omega)).
\end{equation}
The goal is to reconstruct the state at the final time $T_0$, i.e.,
$U(T_0) = (u(T_0), p(T_0))$, without knowledge of $U_0$.

For the reconstruction of \(U(T_0)\) we will introduce a control problem for the backward adjoint system: Given a \textit{control} $g \in L^2(0,T_0; L^2(\omega))$ and a terminal datum
$\Phi_0 = (\phi_0, \phi_0|_\Gamma) \in \mathcal{H}$, let us consider the following equation:
\begin{equation}
\label{eq:adjoint}
\begin{cases}
-\partial_t \xi - \Delta \xi = - g\, \mathbf{1}_\omega
& \text{in } \Omega \times (0, T_0), \\
-\beta\, \partial_t \pi + \partial_n \xi - \Delta_\Gamma \pi = 0
& \text{on } \Gamma \times (0, T_0), \\
\xi|_\Gamma = \pi & \text{on } \Gamma \times (0, T_0), \\
\xi(T_0) = \phi_0,\ \ \pi(T_0) = \phi_0|_\Gamma.
\end{cases}
\end{equation}

\begin{remark}
The change of variable $t \mapsto T_0 - t$ converts \eqref{eq:adjoint} into a forward bulk--surface problem with the same structure as \eqref{eq:state} and right-hand side $ g\mathbf{1}_\omega$ in the bulk; existence and uniqueness in $C([0, T_0]; \mathcal{H})$ follow.
\end{remark}

\subsection{The Carleman estimate, observability and null controllability of the adjoint}\label{ssec:carleman}

The cornerstone of the analysis is the global Carleman estimate of
Maniar, Meyries and Schnaubelt \cite[Lemma~3.2]{MR3669656}, established
there for the backward adjoint problem associated with bulk--surface
parabolic equations of surface-diffusion type. We use it with the
coefficients $d=1$, $\delta=\beta$, $a=b=0$, $\Gamma$ playing the role of
their reactive--diffusive boundary, and we retain only the two
zeroth-order terms of their full left-hand side (all terms there being
non-negative). With the weight functions
\begin{equation*}
\theta_{\lambda}(t, x) := \frac{e^{\lambda (m\|\psi\|_\infty + \psi(x))}}
{t(T_0 - t)},
\qquad
\eta_{\lambda}(t, x) := \frac{e^{2\lambda m \|\psi\|_\infty}
- e^{\lambda(m\|\psi\|_\infty + \psi(x))}}{t(T_0 - t)},
\end{equation*}
where $\psi$ is the function $\eta^0$ of \cite[Lemma~3.1]{MR3669656}
($\psi>0$ in $\Omega$, $\psi=0$ on $\Gamma$, $|\nabla\psi|>0$ outside
$\omega_0\subset\subset\omega$, $\partial_n\psi\le-c<0$ on $\Gamma$) and
$m>1$, there exist $s_0,\lambda_0\ge1$ and $C>0$ such that for all
$s\ge s_0$, $\lambda\ge\lambda_0$ and every solution $\Xi=(\xi,\pi)$ of
\eqref{eq:adjoint} with $g\equiv0$,
\begin{multline}
\label{eq:carleman}
s^3 \lambda^4 \int_0^{T_0}\!\!\!\int_\Omega
e^{-2 s \eta_\lambda} \theta_\lambda^{3} |\xi|^2 \,dx\, dt
+ s^3 \lambda^3 \int_0^{T_0}\!\!\!\int_\Gamma
e^{-2 s \eta_\lambda} \theta_\lambda^{3} |\pi|^2 \,dS\, dt \\
\leq C\, s^3 \lambda^4 \int_0^{T_0}\!\!\!\int_\omega
e^{-2 s \eta_\lambda} \theta_\lambda^{3} |\xi|^2 \,dx\, dt .
\end{multline}

The key feature of \eqref{eq:carleman} for data assimilation is that the right-hand side only contains integrals on $\omega \times (0,T_0)$ plus the bulk forcing term, which vanishes for the homogeneous adjoint $g \equiv 0$.

A standard cut-off argument applied to \eqref{eq:carleman} (see, e.g., \cite{MR2804643}) yields the following:
\begin{proposition}[Observability] \label{prop:obs}
Let $\Xi = (\xi, \pi)$ solve \eqref{eq:adjoint} with $g \equiv 0$ and
terminal data $\Phi_0 \in \mathcal{H}$. There exists $C_{\rm obs} > 0$,
depending only on $\Omega, \omega, \beta, T_0$, such that
\begin{equation}
\label{eq:obs_ineq}
\|\Xi(0)\|_{\mathcal{H}}^2
\leq C_{\rm obs} \int_0^{T_0}\!\!\!\int_\omega |\xi|^2\, dx\, dt.
\end{equation}
\end{proposition}

We can now drive the adjoint solution at $t = 0$ to zero through a suitable control $g$:

\begin{proposition}[Null controllability]
\label{prop:null}
For every $\Phi_0 \in \mathcal{H}$, there exists $g^*(\Phi_0) \in L^2(0,T_0; L^2(\omega))$ of minimal $L^2$-norm such that the solution $\Xi^* = (\xi^*, \pi^*)$ of \eqref{eq:adjoint} with control $g^*$ satisfies $\Xi^*(0) = 0$ in $\mathcal{H}$. Moreover, there exists $C_{\rm null} > 0$ such that
\begin{equation}
\label{eq:null_bound}
\|g^*(\Phi_0)\|_{L^2(0,T_0; L^2(\omega))} \leq
C_{\rm null}\, \|\Phi_0\|_{\mathcal{H}}.
\end{equation}
\end{proposition}

\begin{proof}
For fixed $\Phi_0 \in \mathcal{H}$, define on $\mathcal{H} \times \mathcal{H}$ the bilinear form 
$$a(\widetilde\Phi_0, \widehat\Phi_0) := \int_0^{T_0}\,\,\int_\omega \widetilde\xi \cdot \widehat\xi\, dx\, dt,$$ 
where $\widetilde\xi$ and $\widehat\xi$ are solutions of the homogeneous adjoint problem ($g \equiv 0$) with terminal data $\widetilde\Phi_0$ and $\widehat\Phi_0$.

Let $\mathcal{X}$ be the completion of $\mathcal{H}$ for the norm $\|\widetilde\Phi_0\|_{\mathcal{X}}^2 := a(\widetilde\Phi_0,\widetilde\Phi_0)$. By Proposition~\ref{prop:obs}, $\|\cdot\|_{\mathcal{X}}$ is a genuine (definite) norm and the linear form $\ell(\widetilde\Phi_0) := (\widetilde\Xi(0), \Phi_0)_{\mathcal{H}}$ is continuous on $\mathcal{X}$; the form $a$ is coercive on $\mathcal{X}$ by construction. Lax--Milgram yields a unique $\widehat\Phi_0^* \in \mathcal{X}$ with $a(\widetilde\Phi_0, \widehat\Phi_0^*) = \ell(\widetilde\Phi_0)$ for all $\widetilde\Phi_0$. Setting $g^*(\Phi_0) := \widehat\xi^*|_\omega$, one verifies $\Xi^*(0) = 0$ and that $g^*$ has minimal $L^2$-norm.



Estimate \eqref{eq:null_bound} follows from Proposition~\ref{prop:obs} and the boundedness of $\ell$.
\end{proof}

\subsection{Reconstruction formula and the operator $\Lambda$} \label{ssec:repr}

\begin{theorem}[Reconstruction formula]
\label{thm:repr}
Let $\Phi_0 \in \mathcal{H}$ and let $g^*(\Phi_0)$ be the null-control of Proposition~\ref{prop:null}, with adjoint solution $\Xi^* = (\xi^*, \pi^*)$ satisfying $\Xi^*(0) = 0$. Then, for every solution $U = (u, p)$ of
\eqref{eq:state} with sources $F = (f, f_\Gamma)$,
\begin{equation}
\label{eq:repr}
\bigl( U(T_0), \Phi_0 \bigr)_{\mathcal{H}}
= \int_0^{T_0}\!\!\!\int_\omega h\, g^*(\Phi_0)\, dx\, dt
+ \int_0^{T_0}\!\!\!\int_\Omega f \cdot \xi^*\, dx\, dt
+ \int_0^{T_0}\!\!\!\int_\Gamma f_\Gamma \cdot \pi^*\, dS\, dt.
\end{equation}
The right-hand side does \emph{not} involve the unknown initial datum
$U_0$.
\end{theorem}

\begin{proof}
Multiply the first equation of \eqref{eq:state} by $\xi^*$ and integrate over $\Omega \times (0, T_0)$. The time term becomes, after integration by parts,
\begin{equation}
\label{eq:proof_time}
\int_0^{T_0}\!\!\!\int_\Omega \partial_t u \cdot \xi^*\, dx\, dt
= \int_\Omega u(T_0) \phi_0\, dx
- \int_0^{T_0}\!\!\!\int_\Omega u\, \partial_t \xi^*\, dx\, dt,
\end{equation}
where we used $\xi^*(T_0) = \phi_0$ and $\xi^*(0) = 0$. The Laplacian term, by Green's identity, becomes
\begin{equation}
\label{eq:proof_space}
\int_0^{T_0}\!\!\!\int_\Omega (-\Delta u) \xi^*\, dx\, dt
= \int_0^{T_0}\!\!\!\int_\Omega u (-\Delta \xi^*)\, dx\, dt
+ \int_0^{T_0}\!\!\!\int_\Gamma \bigl( u \partial_n \xi^*
- \partial_n u \cdot \xi^* \bigr) dS\, dt.
\end{equation}
Now substitute the boundary equations: from the second equation of \eqref{eq:state}, $\partial_n u = -\beta \partial_t p + \Delta_\Gamma p + f_\Gamma$, and from the second equation of \eqref{eq:adjoint}, $\partial_n \xi^* = \beta \partial_t \pi^* + \Delta_\Gamma \pi^*$. Using $u|_\Gamma = p$, $\xi^*|_\Gamma = \pi^*$, the boundary integral in \eqref{eq:proof_space} becomes
\begin{align*}
&\int_0^{T_0}\!\!\!\int_\Gamma \bigl( u\, \partial_n \xi^*
- \partial_n u \cdot \xi^* \bigr) dS dt \\
&\quad = \int_0^{T_0}\!\!\!\int_\Gamma \bigl[
p (\beta\partial_t \pi^* + \Delta_\Gamma \pi^*)
- (-\beta\partial_t p + \Delta_\Gamma p + f_\Gamma) \pi^* \bigr] dS dt \\
&\quad = \int_0^{T_0}\!\!\!\int_\Gamma
\beta\, \partial_t (p\, \pi^*)\, dS\, dt
+ \int_0^{T_0}\!\!\!\int_\Gamma
( p\, \Delta_\Gamma \pi^* - \Delta_\Gamma p \cdot \pi^* )\, dS\, dt
- \int_0^{T_0}\!\!\!\int_\Gamma f_\Gamma\, \pi^*\, dS\, dt.
\end{align*}
The middle integral vanishes by Green's identity on the closed surface $\Gamma$. Integrating $\partial_t(p\, \pi^*)$ in $t$ and using $\pi^*(0) = 0$, $\pi^*(T_0) = \phi_0|_\Gamma$ gives 
$$\int_0^{T_0}\!\int_\Gamma \beta \partial_t(p \pi^*) dS dt = \beta \int_\Gamma p(T_0) \phi_0|_\Gamma\, dS.$$
Combining \eqref{eq:proof_time} and \eqref{eq:proof_space} into
$$\int_0^{T_0}\!\int_\Omega (\partial_t u - \Delta u) \xi^* = \int_0^{T_0}\! \int_\Omega f \xi^*$$ 
yields 
\begin{multline*}
\int_\Omega u(T_0) \phi_0\, dx
+ \beta \int_\Gamma p(T_0) \phi_0|_\Gamma\, dS \\
= -\int_0^{T_0}\!\!\!\int_\Omega u (-\partial_t \xi^* - \Delta \xi^*) dx dt
+ \int_0^{T_0}\!\!\!\int_\Omega f \xi^* dx dt
+ \int_0^{T_0}\!\!\!\int_\Gamma f_\Gamma \pi^* dS dt.
\end{multline*}
Finally, $-\partial_t \xi^* - \Delta \xi^* = -g^*\mathbf{1}_\omega$ and $u|_\omega = h$, so the first integral on the right equals 
$$\int_0^{T_0}\!\int_\omega h\, g^* dx dt.$$
The left-hand side is exactly $(U(T_0), \Phi_0)_{\mathcal{H}}$ by \eqref{eq:Hinner} and yields \eqref{eq:repr}.
\end{proof}

\begin{remark}
\label{rmk:basis} 
By taking $\Phi_0$ to run over an orthonormal basis of $\mathcal{H}$---in practice, the eigenpairs of the bulk--surface Laplace operator (Section \ref{sec:discretization})---one recovers the full state $U(T_0)$.
\end{remark}


\begin{definition}[The operator $\Lambda$]
\label{def:Lambda}
For $\Phi_0 \in \mathcal{H}$, let $g^*(\Phi_0) \in L^2(0,T_0; L^2(\omega))$
be the null-control of Proposition~\ref{prop:null}. We define the
\emph{reconstruction operator} $\Lambda : \mathcal{H} \to \mathcal{H}$ as
the Riesz representation of the bounded symmetric bilinear form
\begin{equation}
\label{eq:Lambda_def}
(\Lambda \Phi_0, \Psi_0)_{\mathcal{H}}
:= \int_0^{T_0}\!\!\!\int_\omega g^*(\Phi_0)\, g^*(\Psi_0)\, dx\, dt,
\qquad \Phi_0, \Psi_0 \in \mathcal{H};
\end{equation}
that is, $\Lambda \Phi_0$ is the unique element of $\mathcal{H}$ for which
\eqref{eq:Lambda_def} holds for every $\Psi_0 \in \mathcal{H}$.
\end{definition}

\begin{remark} \label{rmk:Lambda_sym}
Definition \ref{def:Lambda} makes $\Lambda$ symmetric by construction. The forward map sending $\Phi_0$ to the value at $T_0$ of the solution of \eqref{eq:state} with source $g^*(\Phi_0)\mathbf{1}_\omega$, zero surface forcing and zero initial datum, is \emph{not} self-adjoint for partial observation $\omega \subsetneq \Omega$, since the control Gramian and the heat semigroup do not commute. This is why we assemble the discrete object of Section~\ref{sec:discretization} \emph{directly} as a Gram matrix $\boldsymbol{\mathcal{C}}_h$---symmetric positive semi-definite by construction---rather than through that forward map, which would carry an $\mathcal O(1)$ antisymmetric part.
\end{remark}

\begin{proposition}[Properties of $\Lambda$]
\label{prop:Lambda}
$\Lambda : \mathcal{H} \to \mathcal{H}$ is linear, bounded, symmetric and
positive semi-definite; its range is dense in $\mathcal{H}$.
\end{proposition}

\begin{proof}
Boundedness follows from \eqref{eq:null_bound} and Cauchy--Schwarz: for
all $\Phi_0, \Psi_0 \in \mathcal{H}$,
$$
|(\Lambda \Phi_0, \Psi_0)_{\mathcal{H}}|
\leq \|g^*(\Phi_0)\|_{L^2(\omega\times(0,T_0))}\,
     \|g^*(\Psi_0)\|_{L^2(\omega\times(0,T_0))}
\leq C_{\rm null}^2\, \|\Phi_0\|_{\mathcal{H}}\,\|\Psi_0\|_{\mathcal{H}},
$$
hence $\|\Lambda\| \leq C_{\rm null}^2$. Symmetry is immediate, since the
right-hand side of \eqref{eq:Lambda_def} is symmetric in
$(\Phi_0, \Psi_0)$. Taking $\Psi_0 = \Phi_0$ gives
$$
(\Lambda \Phi_0, \Phi_0)_{\mathcal{H}}
= \|g^*(\Phi_0)\|_{L^2(\omega\times(0,T_0))}^2 \geq 0,
$$
so $\Lambda$ is positive semi-definite. It is in fact positive definite:
if $(\Lambda \Phi_0, \Phi_0)_{\mathcal{H}} = 0$ then $g^*(\Phi_0) \equiv 0$,
so the homogeneous adjoint ($g \equiv 0$) with terminal datum $\Phi_0$
satisfies $\Xi^*(0) = 0$. Under the change of variable $t \mapsto T_0 - t$
this is the forward bulk--surface heat semigroup acting on $\Phi_0$ over
time $T_0$, which is injective; therefore $\Phi_0 = 0$. Since $\Lambda$ is
bounded and self-adjoint with trivial kernel,
$\overline{\operatorname{ran}\Lambda}
= (\ker \Lambda)^{\perp} = \mathcal{H}$, i.e.\ the range is dense.
\end{proof}


\subsection{Tikhonov-penalized minimization}
\label{ssec:penalized}

Solving the null-control problem $\Xi(0) = 0$ exactly is impractical at the
discrete level; following \cite{MR2491591,MR2804643} we relax it by
penalization, whose numerical treatment as a variational problem is
classical \cite{MR2423313}. For $\alpha > 0$, define
\begin{equation}
\label{eq:Jalpha}
J_\alpha(g) := \frac{1}{2} \int_0^{T_0}\!\!\!\int_\omega |g|^2\, dx\, dt
+ \frac{1}{2\alpha}\, \|\Xi(0)\|_{\mathcal{H}}^2,
\qquad g \in L^2(0,T_0; L^2(\omega)),
\end{equation}
where $\Xi = \Xi[g,\Phi_0] = (\xi, \pi)$ solves \eqref{eq:adjoint} with
control $g$ and terminal datum $\Phi_0 \in \mathcal{H}$. We write
$\mathcal{U}_{\rm ad}(\Phi_0) := \{ g : \Xi[g,\Phi_0](0) = 0 \}$ for the
(nonempty, by Proposition~\ref{prop:null}) set of exact controls.

\paragraph{Operator setting.}
To exploit convex duality we recast \eqref{eq:Jalpha} through the linear
control-to-state structure of \eqref{eq:adjoint}. By linearity its trace
at $t = 0$ splits as $\Xi[g,\Phi_0](0) = L g + S \Phi_0$, where
\begin{itemize}
\item $S : \mathcal{H} \to \mathcal{H}$, $S\Phi_0 := \Xi[0,\Phi_0](0)$, is
the homogeneous adjoint propagator. Under $t \mapsto T_0 - t$ it coincides
with the forward bulk--surface semigroup $e^{T_0 \mathcal{A}}$ generated by
the (self-adjoint, dissipative) bulk--surface Laplacian $\mathcal{A}$ on
$\mathcal{H}$; hence $S$ is bounded, self-adjoint and injective;
\item $L : L^2(0,T_0; L^2(\omega)) \to \mathcal{H}$, $L g := \Xi[g,0](0)$,
is the bounded control-to-initial-state operator.
\end{itemize}

\begin{lemma}[Adjoint and Gramian]
\label{lem:Lstar}
For every $\zeta \in \mathcal{H}$, $L^* \zeta = - y|_\omega$, where
$Y = (y, \cdot)$ is the source-free forward solution of \eqref{eq:state}
($f = 0$, $f_\Gamma = 0$) with initial datum $Y(0) = \zeta$. Consequently
the controllability Gramian $\mathcal{C} := L L^* : \mathcal{H} \to
\mathcal{H}$ is bounded, self-adjoint and positive semi-definite, with
\begin{equation}
\label{eq:gramian}
(\mathcal{C}\zeta, \zeta)_{\mathcal{H}}
= \int_0^{T_0}\!\!\!\int_\omega |y|^2\, dx\, dt .
\end{equation}
\end{lemma}

\begin{proof}
Let $\delta\Xi := \Xi[g,0]$ (control $g$, zero terminal datum), so
$\delta\Xi(0) = L g$ and $\delta\Xi(T_0) = 0$. Pairing the source-free
forward $Y$ against $\delta\Xi$ exactly as in the proof of
Theorem~\ref{thm:repr} (now with vanishing sources for $Y$ and vanishing
terminal datum for $\delta\Xi$) gives the duality identity
$$
(\zeta, L g)_{\mathcal{H}} = (\zeta, \delta\Xi(0))_{\mathcal{H}}
= - \int_0^{T_0}\!\!\!\int_\omega y\, g\, dx\, dt,
$$
whence $L^* \zeta = - y|_\omega$. Then $(\mathcal{C}\zeta, \zeta) =
\|L^*\zeta\|_{L^2(\omega\times(0,T_0))}^2$, which is \eqref{eq:gramian};
self-adjointness and positivity are immediate.
\end{proof}

\begin{remark}
  In these terms 
$$J_\alpha(g) = \tfrac12 \|g\|_{L^2(\omega\times(0,T_0))}^2
+ \tfrac{1}{2\alpha} \| L g + S\Phi_0 \|_{\mathcal{H}}^2.$$  
\end{remark}

\paragraph{The data functional.}
Let $\Xi_\alpha = \Xi[g_\alpha, \Phi_0] = (\xi_\alpha, \pi_\alpha)$ be the
adjoint generated by the (yet to be defined) penalized control
$g_\alpha = g_\alpha(\Phi_0)$, which depends linearly on $\Phi_0$. For the
data $(h, f, f_\Gamma)$ held fixed, the map
\begin{equation}
\label{eq:Rfunctional}
\Phi_0 \longmapsto
\int_0^{T_0}\!\!\!\int_\omega h\, g_\alpha(\Phi_0)\, dx\, dt
+ \int_0^{T_0}\!\!\!\int_\Omega f\, \xi_\alpha(\Phi_0)\, dx\, dt
+ \int_0^{T_0}\!\!\!\int_\Gamma f_\Gamma\, \pi_\alpha(\Phi_0)\, dS\, dt
\end{equation}
is linear and bounded on $\mathcal{H}$, uniformly in $\alpha$ by
\eqref{eq:null_bound} and parabolic energy estimates. We denote by
$R_\alpha \in \mathcal{H}$ its Riesz representative. By construction
$R_\alpha$ is assembled from the measured data and the computable penalized
controls alone; it does \emph{not} presuppose knowledge of $U(T_0)$.

\begin{theorem}[Penalized reconstruction]
\label{thm:penalized}
Let $\Phi_0 \in \mathcal{H}$ and $\alpha > 0$.
\begin{enumerate}
\item[\textup{(i)}] $J_\alpha$ admits a unique minimizer $g_\alpha \in
L^2(0,T_0; L^2(\omega))$, characterized by the optimality system
\begin{equation}
\label{eq:opt_pair}
g_\alpha = \widetilde\phi_\alpha|_\omega,
\qquad
\alpha\, \widetilde\Phi_\alpha(0) = \Xi_\alpha(0) \quad \text{in }
\mathcal{H},
\end{equation}
where $\Xi_\alpha = \Xi[g_\alpha,\Phi_0]$ and $\widetilde\Phi_\alpha =
(\widetilde\phi_\alpha, \widetilde\pi_\alpha)$ is the source-free forward
solution of \eqref{eq:state} with initial datum $\widetilde\Phi_\alpha(0)$.

\item[\textup{(ii)}] By Fenchel duality, $g_\alpha = - L^* \lambda_\alpha$,
where the dual variable $\lambda_\alpha = \widetilde\Phi_\alpha(0) \in
\mathcal{H}$ is the unique solution of the well-posed normal equation
\begin{equation}
\label{eq:dual_eq}
(\alpha\, \mathrm{Id} + \mathcal{C})\, \lambda_\alpha = S \Phi_0 .
\end{equation}
The operator $\alpha\,\mathrm{Id} + \mathcal{C}$ is self-adjoint and
coercive on $\mathcal{H}$ with coercivity constant $\alpha$.

\item[\textup{(iii)}] As $\alpha \to 0^+$, $g_\alpha \to g^*(\Phi_0)$
strongly in $L^2(0,T_0; L^2(\omega))$, and
$\|\Xi_\alpha(0)\|_{\mathcal{H}} \leq \alpha^{1/2}\, \|g^*(\Phi_0)\|$.

\item[\textup{(iv)}] The reconstruction is consistent: for every
$\Phi_0 \in \mathcal{H}$,
$$
\langle R_\alpha, \Phi_0 \rangle_{\mathcal{H}}
\xrightarrow[\alpha \to 0^+]{} (U(T_0), \Phi_0)_{\mathcal{H}} .
$$
\end{enumerate}
\end{theorem}

\begin{proof}
\emph{(i)} $J_\alpha$ is strictly convex, continuous and coercive on
$L^2(0,T_0;L^2(\omega))$, so the direct method yields a unique minimizer
$g_\alpha$. Its G\^ateaux derivative is
$$
\langle J_\alpha'(g_\alpha), \delta g \rangle
= \int_0^{T_0}\!\!\!\int_\omega g_\alpha\, \delta g\, dx\, dt
+ \tfrac{1}{\alpha}\, (\Xi_\alpha(0), \delta\Xi(0))_{\mathcal{H}},
$$
with $\delta\Xi = \Xi[\delta g, 0]$. By Lemma~\ref{lem:Lstar} applied with
$\zeta = \Xi_\alpha(0)$,
$$
(\Xi_\alpha(0), \delta\Xi(0))_{\mathcal{H}}
= - \int_0^{T_0}\!\!\!\int_\omega y\, \delta g\, dx\, dt,
\qquad Y(0) = \Xi_\alpha(0).
$$
Setting $\widetilde\Phi_\alpha := \alpha^{-1} Y$ (so $\alpha\,
\widetilde\Phi_\alpha(0) = \Xi_\alpha(0)$ and $y = \alpha\,
\widetilde\phi_\alpha$), the Euler--Lagrange equation
$J_\alpha'(g_\alpha) = 0$ becomes $g_\alpha = \widetilde\phi_\alpha|_\omega$,
i.e.\ \eqref{eq:opt_pair}.

\emph{(ii)} Write $J_\alpha(g) = F(g) + G(L g)$ with $F(g) = \tfrac12
\|g\|^2$ and $G(q) = \tfrac{1}{2\alpha} \| q + S\Phi_0\|_{\mathcal{H}}^2$.
Since $F, G$ are convex, continuous and the problem is unconstrained,
Fenchel--Rockafellar duality holds with no gap, and the unique minimizer
$g_\alpha$ satisfies $g_\alpha = - L^* \lambda_\alpha$, where the dual
optimum $\lambda_\alpha \in \mathcal{H}$ minimizes the (strictly convex)
dual functional
$$
\mathcal{D}_\alpha(\lambda) := \tfrac12 \| L^* \lambda \|_{L^2}^2
+ \tfrac{\alpha}{2} \|\lambda\|_{\mathcal{H}}^2
- (\lambda, S\Phi_0)_{\mathcal{H}} .
$$
Its Euler equation is precisely \eqref{eq:dual_eq}. Substituting
$g_\alpha = - L^*\lambda_\alpha$ into \eqref{eq:opt_pair} identifies
$\lambda_\alpha = \widetilde\Phi_\alpha(0)$: indeed $L^* \lambda_\alpha
= - \widetilde\phi_\alpha|_\omega$ by Lemma~\ref{lem:Lstar}, so
$g_\alpha = \widetilde\phi_\alpha|_\omega$ as in (i). Since $\mathcal{C}$ is
positive semi-definite, $\alpha\,\mathrm{Id} + \mathcal{C}$ is self-adjoint
and $\geq \alpha\,\mathrm{Id}$; Lax--Milgram gives the unique solution of
\eqref{eq:dual_eq}.

\emph{(iii)} The exact null-control $g^*(\Phi_0) \in
\mathcal{U}_{\rm ad}(\Phi_0)$ has zero penalty in \eqref{eq:Jalpha}
(since $\Xi^*(0) = 0$), so $J_\alpha(g_\alpha) \leq J_\alpha(g^*) =
\tfrac12 \|g^*\|^2$. This yields both $\|g_\alpha\| \leq \|g^*\|$ and
$\tfrac{1}{\alpha}\|\Xi_\alpha(0)\|_{\mathcal{H}}^2 \leq \|g^*\|^2$, the
latter being the stated rate. Up to a subsequence $g_\alpha \rightharpoonup
\tilde g$; by continuous dependence $\Xi_\alpha(0) \to 0$, so $\tilde g \in
\mathcal{U}_{\rm ad}(\Phi_0)$, and lower semicontinuity gives
$\|\tilde g\| \leq \liminf \|g_\alpha\| \leq \|g^*\|$. By minimality of
$g^*$, $\tilde g = g^*$, and norm convergence upgrades weak to strong
convergence.

\emph{(iv)} Subtracting \eqref{eq:Rfunctional} from the exact
representation \eqref{eq:repr},
$$
(U(T_0), \Phi_0)_{\mathcal{H}} - \langle R_\alpha, \Phi_0 \rangle
= \int_0^{T_0}\!\!\!\int_\omega h\, (g^* - g_\alpha)
+ \int_0^{T_0}\!\!\!\int_\Omega f\, (\xi^* - \xi_\alpha)
+ \int_0^{T_0}\!\!\!\int_\Gamma f_\Gamma\, (\pi^* - \pi_\alpha) .
$$
By (iii), $g_\alpha \to g^*$ in $L^2(\omega\times(0,T_0))$, and by
continuous dependence of \eqref{eq:adjoint} on its control,
$(\xi_\alpha, \pi_\alpha) \to (\xi^*, \pi^*)$ in
$C([0,T_0];\mathcal{H})$. Each term vanishes as $\alpha \to 0^+$.
\end{proof}

Theorem~\ref{thm:penalized} prescribes a constructive scheme: for each
$\Phi_k$ in an $\mathcal{H}$-orthonormal basis, solve the well-posed normal
equation \eqref{eq:dual_eq} for $\lambda_\alpha(\Phi_k)$, recover the
control $g_\alpha(\Phi_k) = - L^* \lambda_\alpha(\Phi_k) =
\widetilde\phi_\alpha|_\omega$, and read off the coefficient
$\langle R_\alpha, \Phi_k \rangle$ from \eqref{eq:Rfunctional}. As
$\alpha \to 0$, the assembled state $\sum_k \langle R_\alpha, \Phi_k\rangle
\Phi_k$ converges to $U(T_0)$.

\begin{remark}
\label{rmk:exact_vs_pen}
The exact reconstruction operator of Definition~\ref{def:Lambda} factors as $\Lambda = S\,\mathcal{C}^{-1} S$, with $\mathcal{C}^{-1}$ well-defined on $\operatorname{ran}(S)$ by the observability estimate of Proposition~\ref{prop:obs}. The penalized scheme regularizes the unbounded inverse, $\mathcal{C}^{-1} \mapsto (\alpha\,\mathrm{Id} + \mathcal{C})^{-1}$, which is exactly \eqref{eq:dual_eq}; correspondingly $g_\alpha = - L^* (\alpha\,\mathrm{Id} + \mathcal{C})^{-1} S\Phi_0 \to g^*(\Phi_0)$. This is the abstract counterpart of assembling $\boldsymbol{\mathcal{C}}_h$ directly as a Gram matrix and tempering it by the penalty $\alpha I$ in Section~\ref{sec:discretization}.
\end{remark}

\section{Discretization}\label{sec:discretization}

We discretize \eqref{eq:state} by combining (i) a $P_1$ finite-element spatial discretization on a regular triangulation of $\Omega$ with mesh size $h$; (ii) a Lie splitting in time with step $\tau$, separating the bulk diffusion step from the surface dynamics, following the framework of \cite{MR4568436}; (iii) a Galerkin reduction onto the first $N_E$ eigenmodes of the bulk--surface Laplacian.

\subsection{Spatial discretization} \label{ssec:disc_space}

Let $\mathcal{T}_h$ be a shape-regular triangulation of $\overline{\Omega}$ with mesh size $h$, and let $V_h \subset H^1(\Omega)$ be the standard $P_1$ Lagrange space. We split the nodes of $\mathcal{T}_h$ into interior nodes $\mathcal{N}_\Omega^{\rm int}$ and boundary nodes $\mathcal{N}_\Gamma$, of cardinalities $n_\Omega^{\rm int}$ and $n_\Gamma$ respectively. The trace coupling $u|_\Gamma = p$ implies that the surface DOFs $p_h$ coincide with the bulk DOFs restricted to $\mathcal{N}_\Gamma$; we therefore use the augmented block vector 
\begin{equation*}
W_h^n = (u_h^n, p_h^n) \in \R^{n_\Omega^{\rm int}} \times \R^{n_\Gamma},
\end{equation*}
with the implicit identification $p_h^n = u_h^n|_{\mathcal{N}_\Gamma}$. We introduce the mass and stiffness matrices in block form
\begin{equation}
\label{eq:matrices}
\mathbf{M} \,=\, \begin{pmatrix} M_\Omega & 0 \\ 0 & \beta\, M_\Gamma \end{pmatrix},
\qquad
\mathbf{K} \,=\, \begin{pmatrix} K_\Omega & K_{\Omega\Gamma} \\
K_{\Omega\Gamma}^{\!\top} & K_\Gamma \end{pmatrix},
\end{equation}
where $M_\Omega, K_\Omega$ are the standard bulk FE matrices, $M_\Gamma, K_\Gamma$ are their surface analogues (involving $\Delta_\Gamma$), and $K_{\Omega\Gamma}$ encodes the bulk--surface coupling through the conormal derivative. The factor $\beta$ in $\mathbf{M}$ reflects the energy inner product \eqref{eq:Hinner}: $W \mapsto W^\top \mathbf{M} W$ is the discrete analogue of $\|U\|_{\mathcal{H}}^2$.

The semi-discrete problem reads
\begin{equation}
\label{eq:semidiscrete}
\mathbf{M}\, \dot W_h(t) + \mathbf{K}\, W_h(t) = F_h(t),
\qquad W_h(0) = W_{h,0},
\end{equation}
with $F_h(t) := (M_\Omega f_h(t), M_\Gamma f_{\Gamma,h}(t))$.

\subsection{Lie splitting in time} \label{ssec:lie}

To decouple the bulk and surface stiffness contributions while preserving the implicit treatment of diffusion, we split $\mathbf{K} = \mathbf{K}_\Omega + \mathbf{K}_\Gamma$, where \begin{equation*}
\mathbf{K}_\Omega \,=\, \begin{pmatrix} K_\Omega & K_{\Omega\Gamma} \\
K_{\Omega\Gamma}^{\!\top} & 0 \end{pmatrix},
\qquad
\mathbf{K}_\Gamma \,=\, \begin{pmatrix} 0 & 0 \\ 0 & K_\Gamma \end{pmatrix},
\end{equation*}
and apply a first-order operator splitting
\cite{MR4568436}: given $W_h^n$, compute the half-step $W_h^{n+1/2}$ by integrating the bulk part with backward Euler over $[t^n, t^{n+1}]$,
\begin{equation}
\label{eq:lie_A}
(\mathbf{M} + \tau \mathbf{K}_\Omega)\, W_h^{n+1/2} = \mathbf{M}\, W_h^n + \tau\, F_{\Omega,h}^n,
\end{equation}
followed by the surface part,
\begin{equation}
\label{eq:lie_B}
(\mathbf{M} + \tau \mathbf{K}_\Gamma)\, W_h^{n+1}
= \mathbf{M}\, W_h^{n+1/2} + \tau\, F_{\Gamma,h}^n,
\end{equation}
where $F_{\Omega,h}^n = (M_\Omega f_h^n, 0)$ and $F_{\Gamma,h}^n = (0, M_\Gamma f_{\Gamma,h}^n)$. We denote the resulting one-step propagator by 
$$\Psi_\tau: \R^{n_\Omega^{\rm int} + n_\Gamma} \to \R^{n_\Omega^{\rm int} + n_\Gamma},\quad \Psi_\tau W_h^n := W_h^{n+1},$$ 
and the discrete trajectory by $W_h^n = \Psi_\tau^n W_{h,0}$ (with the source absorbed in the inhomogeneous form when $F_h \not\equiv 0$).

\begin{proposition}[Convergence of Lie splitting]
\label{prop:lie_conv}
Under the regularity assumptions of \cite{MR4568436}, the splitting scheme \eqref{eq:lie_A}--\eqref{eq:lie_B} converges to the solution of \eqref{eq:semidiscrete}: there exists $C > 0$ independent of $\tau, h$ such that for $t^n = n\tau \leq T_0$,
\begin{equation}
\label{eq:lie_error}
\|W_h(t^n) - W_h^n\|_{\mathbf{M}} \leq C\, \tau\, (1 + \|W_{h,0}\|_{\mathbf{K}}).
\end{equation}
\end{proposition}

This is \cite{MR4568436} adapted to our notation. We verified \eqref{eq:lie_error} numerically and observe the predicted linear convergence in $\tau$.

\subsection{Bulk--surface eigenvalue problem and reduced basis} \label{ssec:eigen}

The discrete Gramian $\boldsymbol{\mathcal C}_h$ is most naturally represented in the basis of eigenvectors of the bulk--surface Laplacian. Consider the generalized eigenvalue problem
\begin{equation}
\label{eq:eigenproblem}
\mathbf{K}\, \Phi_k \,=\, \mu_k\, \mathbf{M}\, \Phi_k,
\qquad k = 0, 1, 2, \dots,
\end{equation}
with $0 = \mu_0 < \mu_1 \leq \mu_2 \leq \cdots$ (the zero eigenvalue corresponds to constants, which we skip). The eigenvectors are $\mathbf{M}$-orthonormal: $\Phi_k^{\!\top} \mathbf{M}\, \Phi_j = \delta_{kj}$. Truncating to the first $N_E \geq 1$ non-zero modes, we form the reduced basis matrix
\begin{equation*}
\mathbf{V} := [\Phi_1, \dots, \Phi_{N_E}]
\in \R^{(n_\Omega^{\rm int} + n_\Gamma) \times N_E},
\qquad \mathbf{V}^{\!\top} \mathbf{M}\, \mathbf{V} = I_{N_E}.
\end{equation*}
Any reduced-basis vector $\mathbf{c} \in \R^{N_E}$ lifts to a full-space vector $W = \mathbf{V} \mathbf{c}$.

\begin{remark}
\label{rmk:eigen_conv}
For a $C^{1,1}$ domain $\Omega$, the eigenvalues $\mu_k^h$ of the discrete problem \eqref{eq:eigenproblem} converge to those of the continuous bulk--surface Laplacian with rate $h^2$ on the lower modes \cite{MR2743037}. The truncation to $N_E$ modes introduces a spectral truncation error of order $N_E^{-s/d}$, where $s$ measures the Sobolev regularity of $U(T_0)$; see Section \ref{ssec:err_analysis}.
\end{remark}

\subsection{The discrete controllability Gramian \texorpdfstring{$\boldsymbol{\mathcal{C}}_h$}{Ch}} \label{ssec:Lambda_h}

Following the dual formulation of Theorem~\ref{thm:penalized}, the only operator that must be assembled is the controllability Gramian $\mathcal{C} = L L^*$ of Lemma~\ref{lem:Lstar}, represented on the reduced basis. We build it directly as a Gram matrix, which makes it symmetric and positive semi-definite \emph{by construction}.

For each mode $k = 1, \dots, N_E$ we perform a single discrete solve. Let $Y_h^{(k)} = (y_h^{(k)}, q_h^{(k)})$ be the \emph{source-free} forward solution of \eqref{eq:lie_A}--\eqref{eq:lie_B} ($f = 0$, $f_\Gamma = 0$) with initial datum $Y_h^{(k)}(0) = \Phi_k$, evaluated at $t^n = n\tau$, $n = 0, \dots, N$. By Lemma~\ref{lem:Lstar} its bulk trace on the observed nodes is the discrete adjoint action,
\begin{equation}
\label{eq:Lstar_disc}
L_h^* \Phi_k = -\, y_h^{(k)}\big|_\omega \in \R^{n_\omega \times (N+1)} .
\end{equation}
The Gramian entries are then the space--time $M_\omega$ inner products of these traces,
\begin{equation}
\label{eq:Cgram}
\bigl[ \boldsymbol{\mathcal{C}}_h \bigr]_{j,k}
:= \tau \sum_{n=0}^{N-1}
\bigl\langle\, y_h^{(j),n},\, y_h^{(k),n} \,\bigr\rangle_{M_\omega},
\qquad j, k = 1, \dots, N_E,
\end{equation}
where $M_\omega$ is the mass matrix restricted to the observed nodes. By \eqref{eq:Cgram}, $\boldsymbol{\mathcal{C}}_h = G^{\!\top} G$ for a suitable factor $G$, hence $\boldsymbol{\mathcal{C}}_h = \boldsymbol{\mathcal{C}}_h^{\!\top} \succeq 0$ exactly, in any arithmetic up to round-off.

\begin{remark}[Semi-discrete closed form, as a verification]
\label{rmk:Cgram_check}
Before time splitting, $y_h^{(k)}(t) = e^{-\mu_k t}\Phi_k$ in the $\mathbf{M}$-orthonormal eigenbasis \eqref{eq:eigenproblem}, so \eqref{eq:Cgram} reduces to
\begin{equation*}
\bigl[ \boldsymbol{\mathcal{C}} \bigr]_{j,k} = \frac{1 - e^{-(\mu_j + \mu_k) T_0}}{\mu_j + \mu_k}\, (\Phi_j, \Phi_k)_{L^2(\omega)} .
\end{equation*}
We use this expression to validate the assembled $\boldsymbol{\mathcal{C}}_h$ against the splitting solver: the two agree to $\mathcal{O}(\tau)$, the splitting consistency order
(Proposition~\ref{prop:lie_conv}).
\end{remark}

\paragraph{On a posteriori symmetrization.}
In contrast to a construction based on the forward map $\Phi_k \mapsto Z_h(T_0)$ of Definition~\ref{def:Lambda}'s discarded variant---which is not self-adjoint already at the continuous level (Remark~\ref{rmk:Lambda_sym}) and would carry an $\mathcal{O}(1)$ antisymmetric part---the Gram form \eqref{eq:Cgram} is symmetric intrinsically; the residual asymmetry is at the level of floating-point round-off $\mathcal O(\varepsilon_{\rm mach})$ and we symmetrize it away by the Euclidean average $\tfrac12(\boldsymbol{\mathcal C}_h+\boldsymbol{\mathcal C}_h^\top)$.

\subsection{Discrete penalized reconstruction} \label{ssec:reconstructor_disc}

The discrete scheme mirrors the dual equation \eqref{eq:dual_eq}, not the discarded primal form. We assemble two further reduced-basis objects.

\paragraph{Propagator and data vector.}
Let $\mathbf{S} \in \R^{N_E \times N_E}$ represent the homogeneous adjoint propagator $S$ of Section \ref{ssec:penalized} on the reduced basis, $\mathbf{S}_{j,k} = (\Phi_j, \mathbf{M}\, \Psi_\tau^{N}\!\Phi_k)$; in the eigenbasis it is diagonal up to the splitting error,
$\mathbf{S} = \operatorname{diag}(e^{-\mu_k T_0}) + \mathcal{O}(\tau)$. For the data, we use linearity to subtract the known source contribution: writing $U = U_F + U^{\rm hom}$, where $U_F$ solves \eqref{eq:state} with the (known) sources $F$ and zero initial datum, and $U^{\rm hom}$ carries the unknown initial datum with no source, we set $\tilde z^n := u_{F,h}^n|_\omega$ (computable) and form
\begin{equation}
\label{eq:bvec}
\mathbf{b}_j := \tau \sum_{n=0}^{N-1}
\bigl\langle\, h^n - \tilde z^n,\, y_h^{(j),n} \,\bigr\rangle_{M_\omega},
\qquad j = 1, \dots, N_E,
\end{equation}
with $y_h^{(j)}$ the source-free trajectories of Section \ref{ssec:Lambda_h}. The difference $h^n - \tilde z^n$ is the observation of the homogeneous part $u^{\rm hom}$, consistent with the observation term of \eqref{eq:repr}.

\paragraph{The penalized system.}
Given $\alpha > 0$, the reduced reconstruction is obtained from the single symmetric positive definite system
\begin{equation}
\label{eq:disc_system}
\bigl( \alpha\, I_{N_E} + \boldsymbol{\mathcal{C}}_h \bigr)\, \mathbf{z}
= \mathbf{b},
\qquad
\mathbf{c} = \mathbf{S}^{\!\top} \mathbf{z}.
\end{equation}
Indeed, by the dual equation \eqref{eq:dual_eq} the penalized control of mode $\Phi_k$ has reduced coordinates $\boldsymbol{\lambda}_k = (\alpha I + \boldsymbol{\mathcal{C}}_h)^{-1}\mathbf{S}\,\mathbf{e}_k$, and the reconstruction coefficient
$\mathbf{c}_k = \langle R_\alpha, \Phi_k\rangle = \mathbf{b}^{\!\top} \boldsymbol{\lambda}_k$ of \eqref{eq:Rfunctional}; by symmetry of $\alpha I + \boldsymbol{\mathcal{C}}_h$ this collects into
\eqref{eq:disc_system}. The reconstructed state at $T_0$ is
\begin{equation}
\label{eq:rec_disc}
U_h^{\rm rec}(T_0) = U_{F,h}(T_0) + \mathbf{V} \mathbf{c},
\end{equation}
the first term restoring the known source-driven part.

\paragraph{Cost and conditioning.}
$\boldsymbol{\mathcal{C}}_h$ is a compact-type Gramian: its eigenvalues inherit the decay of the heat semigroup, $\lambda_{\max} = \mathcal{O}(1)$ on the lowest modes (bounded by $C_{\rm null}^2$, cf.\ \eqref{eq:null_bound}) while $\lambda_{\min} \to 0$ on the high modes, which is the analytic signature of the ill-posed inverse. Consequently
\begin{equation*}
\kappa\bigl( \alpha I_{N_E} + \boldsymbol{\mathcal{C}}_h \bigr) = \frac{\alpha + \lambda_{\max}}{\alpha + \lambda_{\min}} = \mathcal{O}\!\left( 1 + \alpha^{-1} \right),
\end{equation*}
so the penalty $\alpha$ is precisely what tempers the conditioning, in the classical Tikhonov trade-off; the system is SPD for every $\alpha > 0$ and no lower bound on $\alpha$ is needed for solvability. We solve \eqref{eq:disc_system} by one Cholesky factorization at cost $\mathcal{O}(N_E^3)$. The dominant cost is the assembly of $\boldsymbol{\mathcal{C}}_h$ via \eqref{eq:Cgram}, requiring only $N_E$ forward time-marches, each $\mathcal{O}(N\,(n_\Omega^{\rm int} + n_\Gamma))$.

\subsection{Adaptive post-processing}

A key observation is that Tikhonov regularization $\mathbf{z} = (\alpha I + \boldsymbol{\mathcal{C}}_h)^{-1}\mathbf{b}$ (whence $\mathbf{c} = \mathbf{S}^{\!\top}\mathbf{z}$) uniformly attenuates all modes, including those whose mass lies inside $\omega$ where the observation $h$ is reliable. The clean-data post-processing step
\begin{equation}
\label{eq:postproc_clean}
u_h^{\rm post}(x) =
\begin{cases}
h(T_0, x), & x \in \omega, \\
(\mathbf{V} \mathbf{c})(x), & x \notin \omega,
\end{cases}
\end{equation}
removes this bias inside $\omega$ at zero cost. Under noise, however, the observation itself is corrupted, and the optimal weight $w_{\rm obs} \in [0,1]$ in
\begin{equation}
\label{eq:postproc_adaptive}
u_h^{\rm post}(x) =
\begin{cases}
w_{\rm obs}\, h(T_0,x) + (1-w_{\rm obs})\,(\mathbf{V}\mathbf{c})(x), & x \in \omega, \\
(\mathbf{V}\mathbf{c})(x), & x \notin \omega,
\end{cases}
\end{equation}
should decrease with the noise level. We propose to estimate the relative noise level from the data via the disagreement at $t = T_0$,
\begin{equation}
\label{eq:eta_hat}
\hat\eta := \frac{\|U_h^{\rm rec}(T_0) - h(T_0)\|_{L^2(\omega)}} {\|h(T_0)\|_{L^2(\omega)}},
\end{equation}
and to select $w_{\rm obs}$ from a three-regime rule:
\begin{equation}
\label{eq:auto_rule}
w_{\rm obs}(\hat\eta) =
\begin{cases}
1.0, & \hat\eta < 0.13, \\
0.5, & 0.13 \leq \hat\eta < 0.20, \\
0.0, & \hat\eta \geq 0.20.
\end{cases}
\end{equation}
The thresholds in \eqref{eq:auto_rule} are calibrated on a single $(r, n)$ configuration; we show in Section \ref{ssec:adaptive} that the resulting estimator matches the oracle (which has access to the true noise level) to within $0.1\,\%$ across more than three orders of magnitude in noise.

\subsection{A priori error bound}\label{ssec:err_analysis}

The total reconstruction error decomposes into four contributions: the penalty bias (Tikhonov), the spectral truncation, the spatial FE discretization, and the temporal splitting. We summarise the resulting bound as a heuristic.

\paragraph{Heuristic a priori error scaling.} 
Let $U(T_0) \in H^s(\Omega) \times H^s(\Gamma)$ for some $s > 0$, and assume a Hölder-type source condition holds: $U(T_0) \in \operatorname{ran}(\boldsymbol{\mathcal{C}}_h^\nu)$ for some $\nu \in (0, 1/2]$ with $\|\boldsymbol{\mathcal{C}}_h^{-\nu} U(T_0)\|_{\mathcal{H}} \leq \rho$. While a fully coupled rigorous error analysis is beyond the scope of this paper, we can establish a heuristic balance for the reconstruction $U_h^{\rm rec}(T_0)$ with clean data and penalty $\alpha > 0$. The total error splits into four standard contributions:

\begin{enumerate}
\item \emph{Penalty bias:} By Theorem~\ref{thm:penalized}(iii) the penalized control satisfies $\|\Xi_\alpha(0)\|_{\mathcal{H}} \leq \alpha^{1/2} \|g^*\|$. Under a Hölder source condition this residual rate transfers to the reconstruction in $\mathcal{H}$ as $\alpha^{1/4}$, the standard qualification-$1/2$ rate for Tikhonov regularization \cite{MR1408680}.
\item \emph{Spectral truncation:} Projection onto the first $N_E$ eigenmodes costs $\mathcal{O}(N_E^{-s/d})$ in $\mathcal{H}$ for $U(T_0) \in H^s$.
\item \emph{Spatial discretization:} TThe $P_1$ energy-norm error is $\mathcal{O}(h^2)$; each source-free trajectory $y_h^{(k)}$, and hence $\boldsymbol{\mathcal{C}}_h$, inherits this rate.

\item \emph{Temporal splitting:} By Proposition~\ref{prop:lie_conv} the splitter is consistent to $\mathcal{O}(\tau)$; propagated through $(\alpha I + \boldsymbol{\mathcal{C}}_h)^{-1}$ with the heuristic stability constant $\alpha^{-1/2}$ this yields $\tau\,\alpha^{-1/2}$.
\end{enumerate}

Summing these four classical rates yields the heuristic a priori bound:
\begin{equation}
\label{eq:err_total}
\| U(T_0) - U_h^{\rm rec}(T_0) \|_{\mathcal{H}}
\;\lesssim\; \alpha^{1/4} + N_E^{-s/d} + h^2 + \tau\, \alpha^{-1/2}.
\end{equation}

\begin{remark}[On the status of \eqref{eq:err_total}]
We stress that \eqref{eq:err_total} is a heuristic balance, not a proven
bound: each term is the textbook rate for the corresponding error source
(Tikhonov qualification-$1/2$ \cite{MR1408680}, spectral truncation,
$P_1$ energy estimate, Lie-splitting consistency of
Proposition~\ref{prop:lie_conv}), but the constants are not tracked across
the coupling and the source condition is assumed, not verified. We retain
\eqref{eq:err_total} only as a guide to the parameter scaling
$\alpha_{\rm opt}\sim\tau^{4/3}$ (Remark~\ref{rmk:beta_opt}), which the
experiments corroborate.
\end{remark}

\begin{remark}
\label{rmk:beta_opt}
The penalty and temporal terms balance when $\alpha^{1/4} \sim \tau\, \alpha^{-1/2}$, i.e.\ $\alpha_{\rm opt} \sim \tau^{4/3}$, at which both are of order $\tau^{1/3}$; choosing $h$ and $N_E$ to match keeps the total at that order. The choice $\alpha_{\rm opt} \sim \tau^{4/3}$ leaves the system well-conditioned, $\kappa = \mathcal{O}(\tau^{-4/3})$, so it is consistent with solvability for every $\alpha > 0$ (no $\alpha \gtrsim \tau$ constraint is needed). For $\tau = 5 \cdot 10^{-3}$ this predicts $\alpha_{\rm opt} \sim 10^{-3}$ and an error floor of order $10\,\%$, consistent with the plateau in test~4 (Section~\ref{ssec:adaptive}).
\end{remark}

\section{Numerical experiments}
\label{sec:numerics}

Throughout this section, $\Omega = (0,1)^2$, $\beta = 1$ (surface diffusion coefficient), $T_0 = 0.1$, $\tau = 5 \cdot 10^{-3}$, $N_E = 30$ (number of spectral modes), and the initial data is $u_0(x,y) = 2\sin(\pi x)\sin(\pi y)$ unless stated otherwise. Twin experiments are used throughout: the "ground truth" is generated by forward simulation, the observation $h$ is extracted on $\omega$, and the reconstruction is compared against the simulated truth at $T_0$ in the mass-weighted norm $\|\cdot\|_{\mathbf{M}}$.

\subsection{Visualization at the optimal setup} \label{ssec:viz}

Before turning to quantitative sensitivity studies, we display in Figure~\ref{fig:viz} the reconstruction obtained at the optimal radius $r^* = 0.21$ (anticipated by Test~1 below) on the refined mesh $n = 32$, with full post-processing ($w_{\rm obs} = 1$) and $\alpha = 3 \cdot 10^{-2}$. The four panels show: (a) the initial datum $u_0^*$, which is unknown to the method; (b) the true final state $u^*(T_0)$, here significantly attenuated by diffusion; (c) the reconstructed state $u_h^{\rm rec}(T_0)$; (d) the pointwise error $|u^*(T_0) - u_h^{\rm rec}(T_0)|$. The post-processed reconstruction is visually indistinguishable from the truth, with relative error $3.30\,\%$ in the $\mathcal{H}$-norm. The error map reveals that the residual is essentially confined to the annular region near $\partial\omega$, consistent with the structural boundary effects of $L^2$ post-processing.

\begin{figure}[htbp]
\centering
\includegraphics[width=0.95\linewidth]{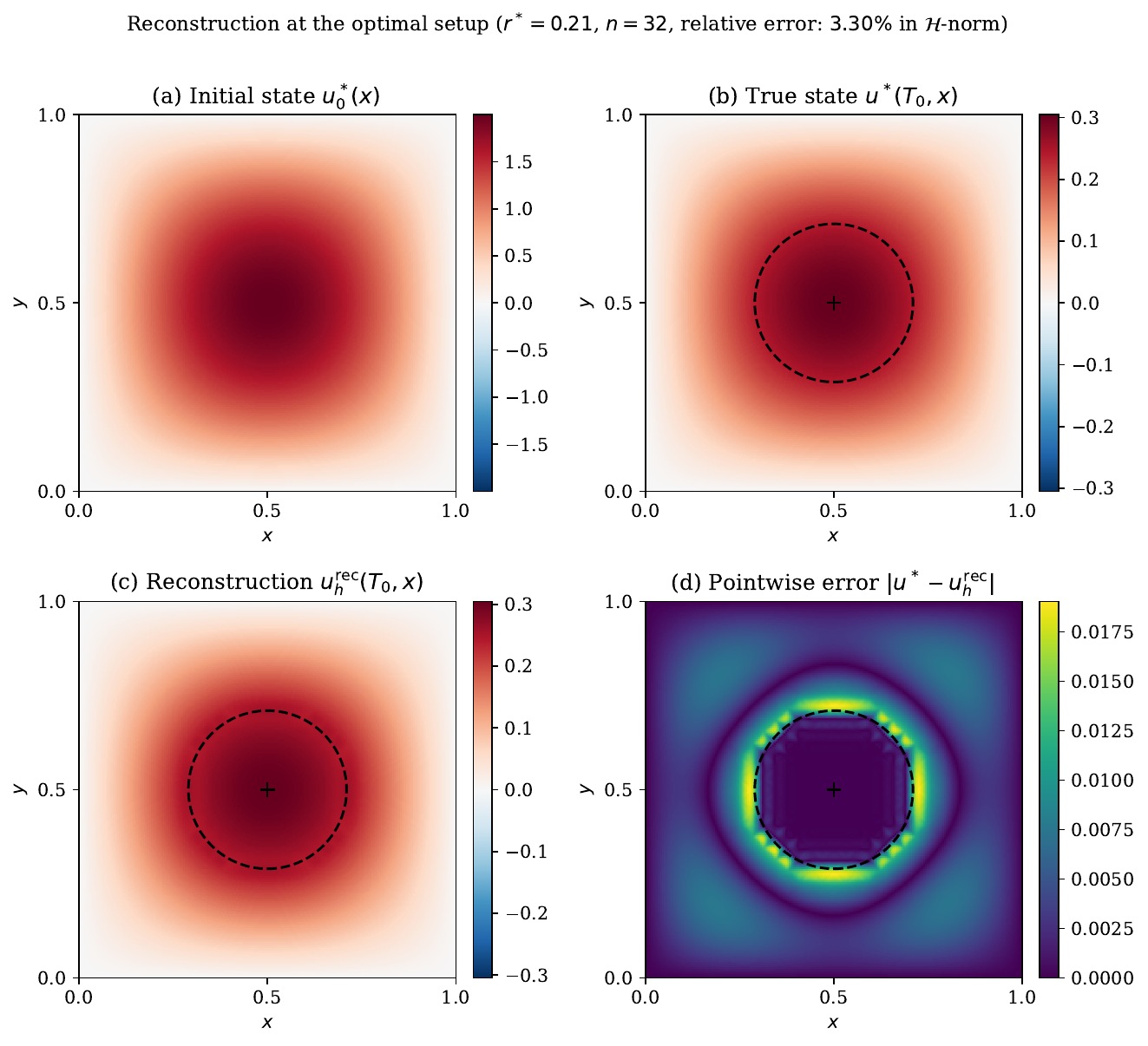}
\caption{Reconstruction at the optimal setup ($r^* = 0.21$, $n = 32$, $T_0 = 0.1$, $\alpha = 3 \cdot 10^{-2}$). (a) Initial datum $u_0^*$. (b) True state $u^*(T_0)$. (c) Reconstruction $u_h^{\rm rec}(T_0)$ with full post-processing. (d) Pointwise error $|u^* - u_h^{\rm rec}|$. The black dashed circle marks the observation domain $\omega$. The relative error in $\mathcal{H}$-norm is $3.30\,\%$.}
\label{fig:viz}
\end{figure}

\begin{remark} \label{rmk:eta_floor}
The same setup with the automatic post-processing rule \eqref{eq:auto_rule} yields a slightly larger error of $5.51\,\%$: the noise estimator \eqref{eq:eta_hat} returns $\hat\eta = 0.131$, so the rule conservatively selects the blend $w_{\rm obs} = 0.5$ rather than full post-processing. This reflects the inherent Tikhonov bias floor of $\hat\eta$ on clean data, whose precise value is configuration-dependent (here $\hat\eta = 0.131$ at $r^* = 0.21$, $n = 32$, versus $\sim 0.11$ for the noise-sweep setup of Section~\ref{ssec:adaptive}):$0.131$ falls just above the lower threshold $0.13$, which triggers the blend. In practice, when a clean-data regime can be identified a priori, the user may override the rule by forcing $w_{\rm obs} = 1$; otherwise, the small extra cost is the price of automatic noise robustness.
\end{remark}

\subsection{Test~1: optimal observatory radius} \label{ssec:radius}

Figure~\ref{fig:radius} reports the relative reconstruction error (post-processed) as a function of the radius of a centered disk observatory. A non-monotone profile is observed: error decreases sharply for $r \in [0.16, 0.21]$, reaches a minimum of $3.3\,\%$ at $r^* \approx 0.21$ on the finer mesh ($n = 32$), then \emph{worsens} until $r \approx 0.30$ before improving again for large radii. The phenomenon is mesh-independent. We attribute the resonance to an alignment of the disk with the support of the dominant eigenmode: the ratio ``information per area'' is maximized at $r^*$.

\begin{figure}[htbp]
\centering
\includegraphics[width=0.85\linewidth]{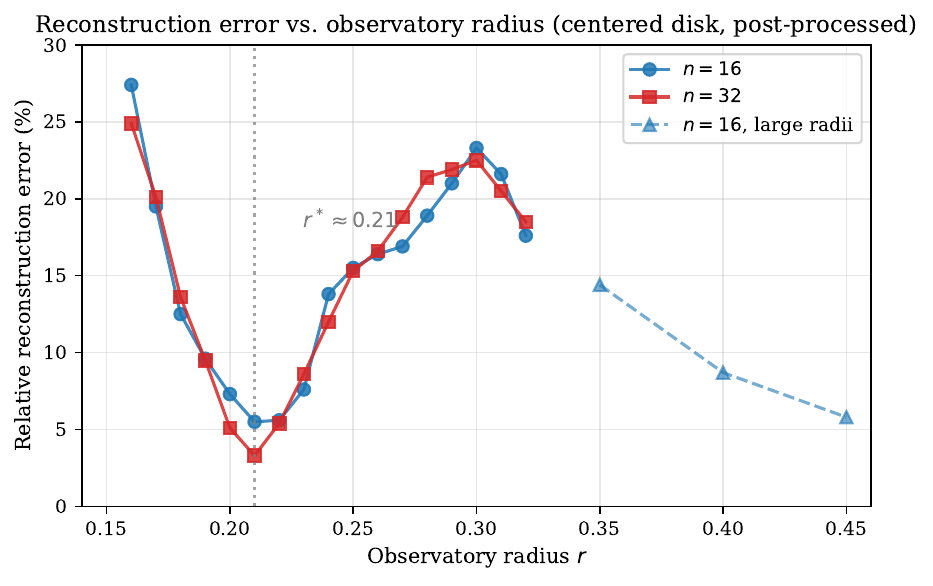}
\caption{Reconstruction error vs.~observatory radius for $\Omega = (0,1)^2$, $T_0 = 0.1$, $\beta = 1$. The error reaches a sharp minimum near $r^* \approx 0.21$ on both meshes, then exhibits a "hump" before recovering at large radii.}
\label{fig:radius}
\end{figure}

\subsection{Test~2: observatory position sensitivity} \label{ssec:position}

Figure~\ref{fig:position} shows that, for a symmetric initial condition, the central position is more than four times better than the corner-ward position with identical observatory size. The result confirms that the information content of $\omega$ is geometry-dependent.

\begin{figure}[htbp]
\centering
\includegraphics[width=0.95\linewidth]{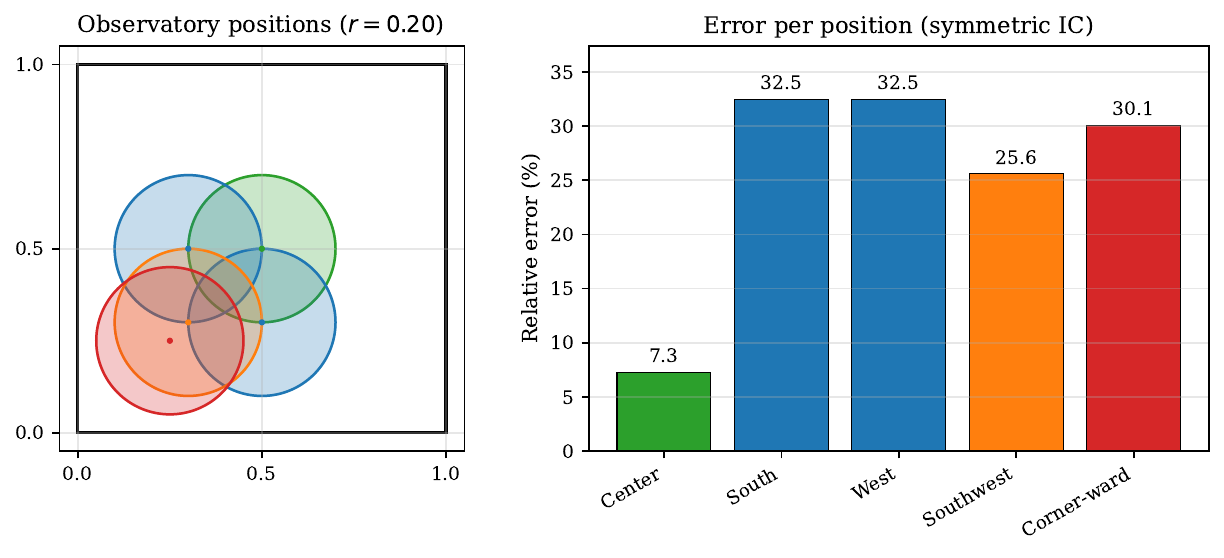}
\caption{Left: five disk positions ($r = 0.20$). Right: corresponding errors
for the symmetric IC.}
\label{fig:position}
\end{figure}

\subsection{Test~3: multi-disk observatory} \label{ssec:multidisk}

We compare four observatory configurations of identical total area ($\sim 28\%$ of $\Omega$): a single disk centred at $(0.5, 0.5)$ of radius $0.30$, and arrays of $2$, $3$, $4$ smaller disks placed so that their total area is preserved. To probe the role of the spectral content of the unknown, we run the experiment with two initial data:
\begin{subequations}\label{eq:ic_two}
\begin{align}
u_0^{\rm sym}(x,y) &= 2\,\sin(\pi x)\sin(\pi y), \label{eq:ic_sym} \\
u_0^{\rm mix}(x,y) &= 1.5\,\sin(\pi x)\sin(\pi y)
                     + 0.5\,\sin(2\pi x)\sin(\pi y)
                     + 0.3\,\sin(\pi x)\sin(2\pi y).
                     \label{eq:ic_mix}
\end{align}
\end{subequations}
The \emph{symmetric} datum~\eqref{eq:ic_sym} consists only of the fundamental Dirichlet mode $(1,1)$, which inherits all four reflection symmetries of the unit square. The \emph{mixed} datum~\eqref{eq:ic_mix} adds two higher harmonics, $(2,1)$ and $(1,2)$, which are antisymmetric under $x \mapsto 1-x$ and $y \mapsto 1-y$ respectively. A single observatory centred at the geometric centre of $\Omega$ is therefore intrinsically blind to these two harmonics (their average over any centro-symmetric region vanishes), whereas a multi-disk observatory breaks the central symmetry and sees them. 

Figure~\ref{fig:multidisk} reports the resulting reconstruction errors. For both initial conditions the four-disk configuration outperforms the single-disk configuration by a factor of $\sim 4.5$ ($5.2$--$5.3\,\%$ versus $23.3$--$23.4\,\%$). The fact that the symmetric and mixed errors agree to within $0.1\,\%$ for every configuration shows that, in this experiment, the spectral-content effect is largely subdominant to the geometric-coverage effect at the resolution chosen. This is a structural advantage of multi-disk observatories: they ``see'' more modes, in line with classical observability theory for control problems \cite{MR3041662,MR2804643}.

\begin{figure}[htbp]
\centering
\includegraphics[width=0.95\linewidth]{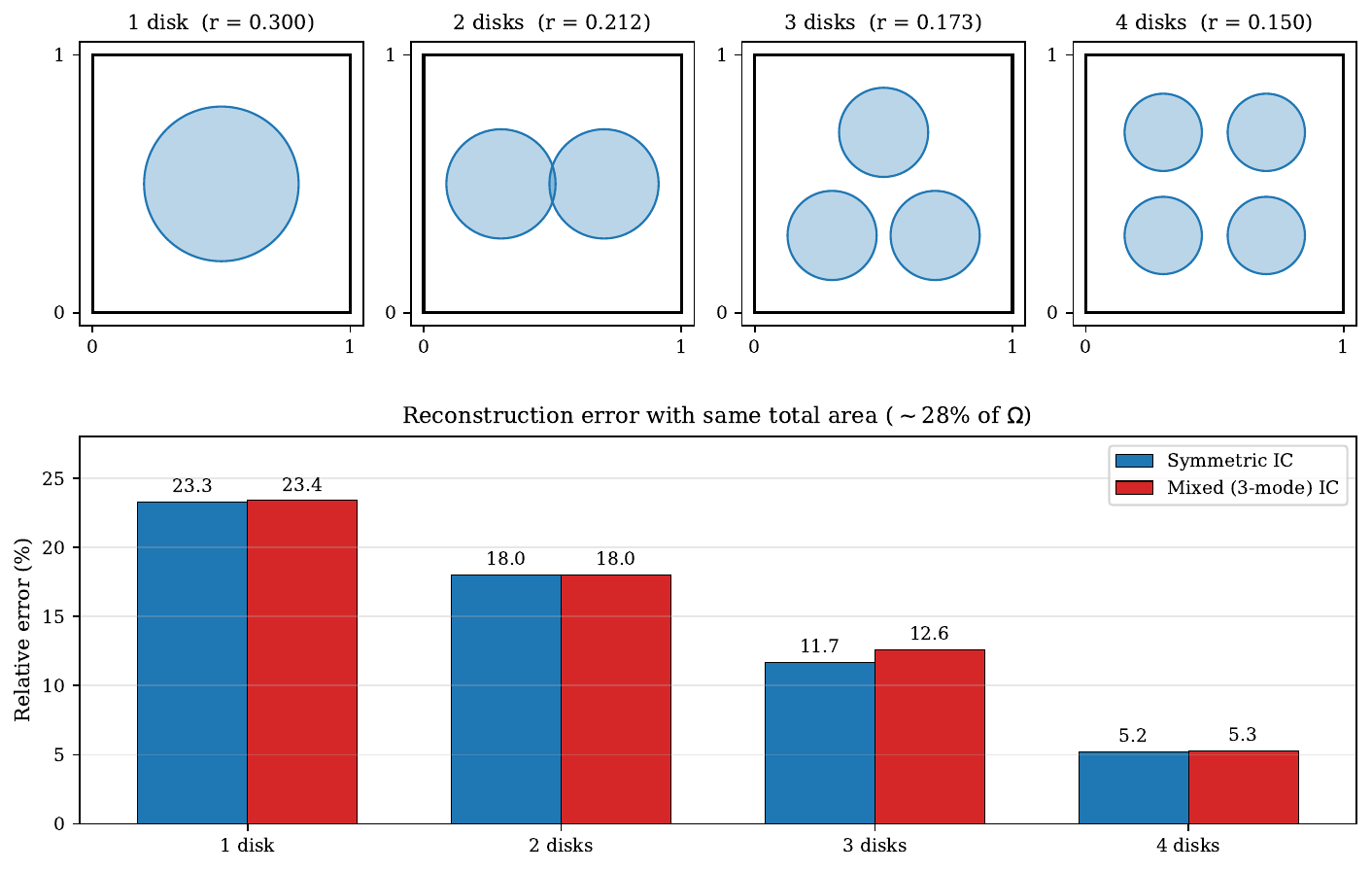}
\caption{Multi-disk geometries (top) and corresponding reconstruction
errors (bottom) for the two initial data~\eqref{eq:ic_sym}--\eqref{eq:ic_mix}.
Same total observatory area ($\sim 28\%$ of $\Omega$) in every
configuration. The \emph{Symmetric IC} bar uses $u_0^{\rm sym}$
(single fundamental mode), while \emph{Mixed (3-mode) IC} uses
$u_0^{\rm mix}$ (three superposed Dirichlet modes).}
\label{fig:multidisk}
\end{figure}

\subsection{Test~4: benefit of clean-data post-processing}
\label{ssec:control}

Figure~\ref{fig:postproc} reports the post-processing benefit across four representative setups: post-processing reduces the relative error by an average factor of $2.2\times$, with the largest gains on larger observatories. The improvement is monotone in the area of $\omega$.

\begin{figure}[htbp]
\centering
\includegraphics[width=0.75\linewidth]{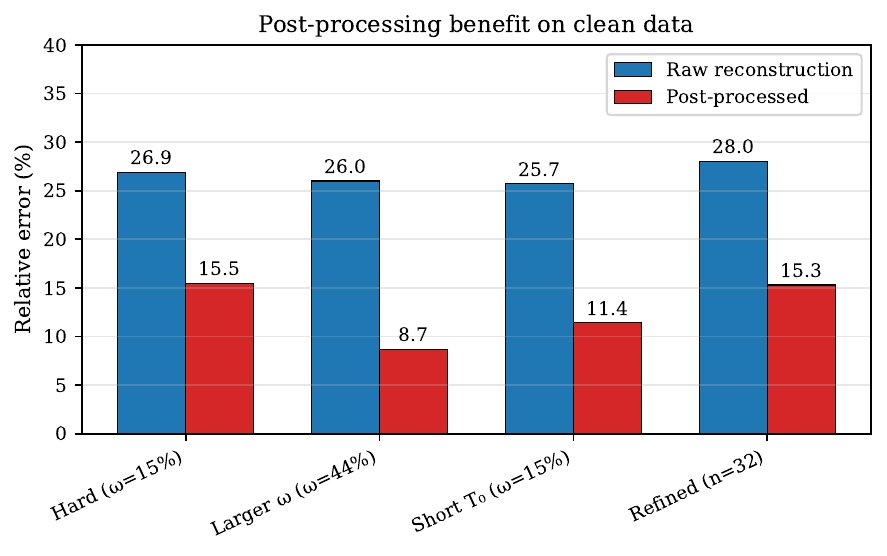}
\caption{Post-processing benefit on clean data, four representative setups.}
\label{fig:postproc}
\end{figure}

\subsection{Test~5 and~6: adaptive post-processing under noise}\label{ssec:adaptive}

To assess robustness, we add zero-mean Gaussian noise to the observation $h$ at each time step, with per-node standard deviation tuned so that the $L^2(\omega \times (0,T_0))$-noise relative level $\eta$ takes prescribed values. Figure~\ref{fig:noise} compares four post-processing strategies across $\eta \in \{0, 10^{-3}, 10^{-2}, 3 \cdot 10^{-2}, 10^{-1}, 0.3, 1.0\}$: $w_{\rm obs} = 0$ (raw), $w_{\rm obs} = 0.5$ (blend), $w_{\rm obs} = 1$ (full post-processing), and the adaptive rule \eqref{eq:auto_rule}. The full post-processing strategy $w_{\rm obs} = 1$ is optimal for clean data but degrades catastrophically as $\eta$ increases (error $\to 250\,\%$ for $\eta = 1$); the raw scheme $w_{\rm obs} = 0$ caps at $\sim 9\,\%$ but is too pessimistic when data are clean. The adaptive rule (green curve) tracks the lower envelope.

\begin{figure}[htbp]
\centering
\includegraphics[width=0.85\linewidth]{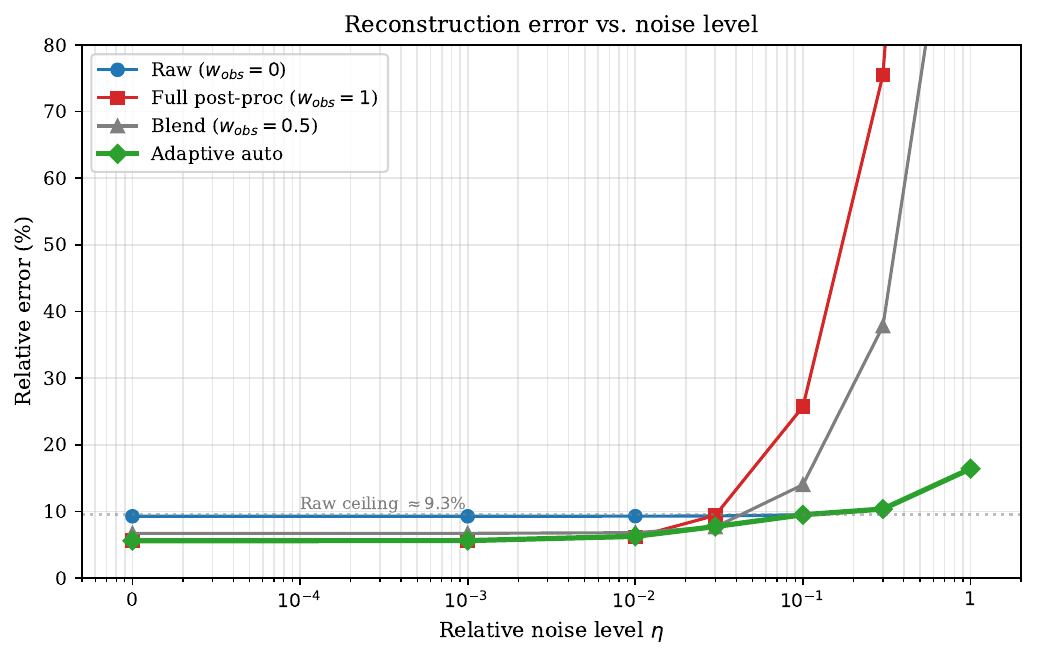}
\caption{Reconstruction error versus noise level for four post-processing strategies. The adaptive scheme automatically tracks the lower envelope.}
\label{fig:noise}
\end{figure}

Figure~\ref{fig:auto} validates the automatic selection rule against the oracle (i.e., the best $w_{\rm obs}$ for each $\eta$ chosen with knowledge of the true error): the two curves are visually indistinguishable. The right panel shows the noise estimator $\hat\eta$ defined in \eqref{eq:eta_hat}. Although $\hat\eta$ has a ``bias floor'' of $\sim 0.11$ inherited from the Tikhonov bias of the reconstruction itself, it correlates linearly with the true noise for $\eta > 5 \cdot 10^{-2}$, which is sufficient for the three-regime rule to discriminate.

\begin{figure}[htbp]
\centering
\includegraphics[width=0.95\linewidth]{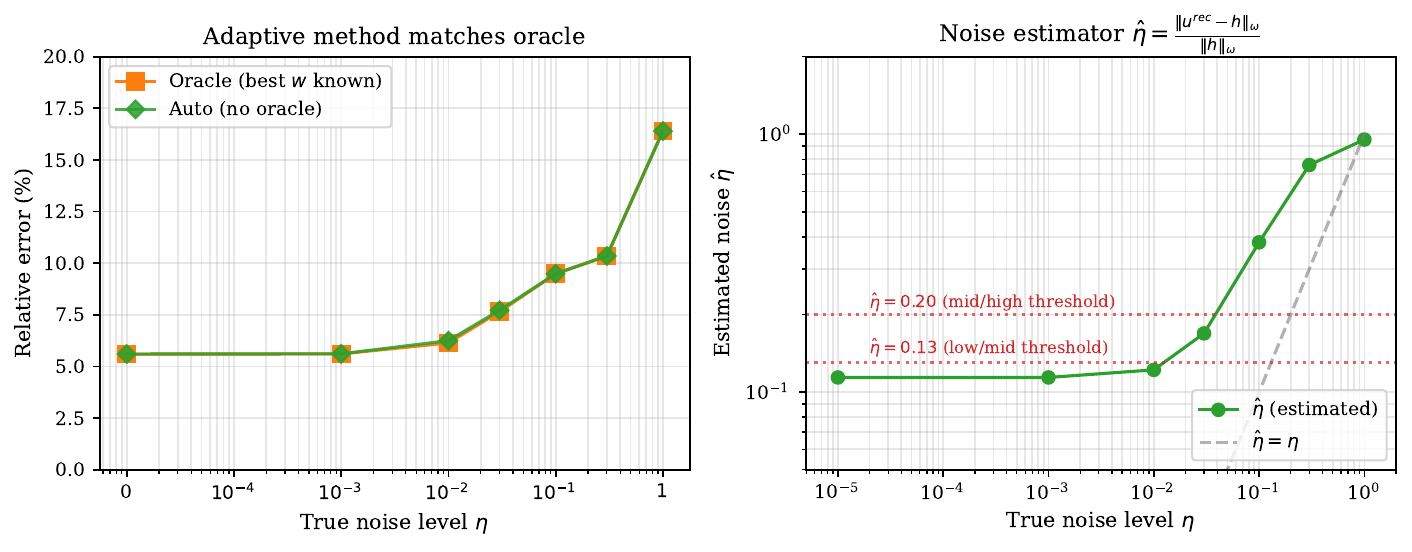}
\caption{Left: oracle vs.\ automatic adaptive method (errors). Right: noise
estimator $\hat\eta$ vs.\ true $\eta$, with threshold lines.}
\label{fig:auto}
\end{figure}

\section{Extension to the semilinear case}\label{sec:semilinear}

In this section we extend the methodology to a class of semilinear bulk--surface problems featuring an Allen--Cahn nonlinearity on the surface. The key observation is that the linear reconstructor of Section~\ref{sec:problem} can be applied as the inner step of a fixed-point iteration on the nonlinear term; under mild Lipschitz assumptions, the iteration converges by a Schauder argument.

\subsection{The semilinear problem}\label{ssec:semilinear_problem}

Replacing the linear surface dynamics in \eqref{eq:state} by an Allen--Cahn--type evolution, we consider
\begin{equation}
\label{eq:state_sl}
\begin{cases}
\partial_t u - \Delta u = f & \text{in } \Omega \times (0, T_0), \\
\beta\, \partial_t p + \partial_n u - \Delta_\Gamma p
= \mathcal{N}(p) + f_\Gamma & \text{on } \Gamma \times (0, T_0), \\
u|_\Gamma = p & \text{on } \Gamma \times (0, T_0),
\end{cases}
\end{equation}
where the surface nonlinearity is the Allen--Cahn double-well derivative
\begin{equation}
\label{eq:AC_phi}
\mathcal{N}(p) := -p^3 + p,
\qquad \mathcal{N}'(p) = -3p^2 + 1,
\end{equation}
arising from the Ginzburg--Landau free energy $W(p) = (p^2 - 1)^2/4$ via $\mathcal{N} = -W'$. The problem \eqref{eq:state_sl} models phase-field dynamics on the boundary of a diffusive medium and has been studied in
\cite{MR2215623,MR3150642,MR1890879}. We assume that the true solution $u^* = (u^*, p^*)$ remains uniformly bounded: there exists $M > 0$ such that
\begin{equation}
\label{eq:Mbound}
\|p^*(t, \cdot)\|_{L^\infty(\Gamma)} \leq M \quad \forall t \in [0, T_0].
\end{equation}
For Allen--Cahn data assimilation in twin experiments, $M$ is determined by the initial data and decreases under the diffusion.

The data assimilation problem is unchanged: given the observation $h = u^*|_{\omega \times (0, T_0)}$, reconstruct $u^*(T_0)$.

\subsection{The Picard outer loop}
\label{ssec:picard}

The key idea is that, when the surface state $p^*$ is treated as a
\emph{known} function $\bar p(t, x)$, the problem \eqref{eq:state_sl}
reduces to the linear problem \eqref{eq:state} with effective surface
source
\begin{equation}
\label{eq:f_eff}
f_\Gamma^{\rm eff}(\bar p) := \mathcal{N}(\bar p) + f_\Gamma
= -\bar p^3 + \bar p + f_\Gamma,
\end{equation}
and the reconstruction Theorems~\ref{thm:repr}--\ref{thm:penalized}
apply directly. This suggests an outer iteration:

\begin{algorithm}[Picard outer loop]\label{alg:picard}
Given an observation $h$, a bulk source $f$, and a surface source $f_\Gamma$, set $\bar p^{(0)} \equiv 0$ on $\Gamma \times (0, T_0)$. For $k = 0, 1, 2, \ldots$ until convergence:
\begin{enumerate}
\item Compute the effective surface source $f_\Gamma^{(k)} := \mathcal{N}(\bar p^{(k)}) + f_\Gamma$ according to \eqref{eq:f_eff}.
\item Apply the linear reconstructor of Section~\ref{sec:problem} with sources $(f, f_\Gamma^{(k)})$ to obtain a reconstruction $u^{(k+1)}(T_0)$ of the final state.
\item Propagate $u^{(k+1)}(T_0)$ \emph{backward} on $(0, T_0)$ using the linear adjoint with the same effective source, yielding a trajectory $\bar u^{(k+1)}(\cdot)$ with surface trace $\bar p^{(k+1)}$.
\end{enumerate}
\end{algorithm}

In Step (3), the backward propagator is the same one constructed in Section \ref{ssec:repr} for the linear case; the procedure produces a full spatio-temporal trajectory consistent with the observation and the current estimate of the nonlinear source. The iteration may be replaced by a Newton step by linearising $\mathcal{N}$ as $\mathcal{N}(p) \approx \mathcal{N}(\bar p) + \mathcal{N}'(\bar p)(p - \bar p)$ and treating the resulting $\mathcal{N}'(\bar p)\, p$ term as a time-dependent reaction coefficient in the linear BC.

\subsection{Convergence: a Schauder fixed-point argument} \label{ssec:schauder}

Define the operator $\mathcal{T} : L^2(0, T_0; L^2(\Gamma)) \to L^2(0, T_0; L^2(\Gamma))$ that, given $\bar p$, returns the surface trace $\bar p^{(k+1)}$ produced by one full pass of Algorithm~\ref{alg:picard}: $\mathcal{T}(\bar p) = \bar p^{(k+1)}$ when $\bar p^{(k)} = \bar p$. A fixed point of $\mathcal{T}$ corresponds to a reconstruction consistent with the nonlinear PDE.

\begin{theorem}[Convergence of the Picard loop]
\label{thm:schauder}
Let $\beta > 0$ be fixed and assume:
\begin{itemize}
\item[\textup{(H1)}] (Boundedness) The true solution satisfies \eqref{eq:Mbound} with some $M < \infty$.
\item[\textup{(H2)}] (Lipschitz growth) The nonlinearity $\mathcal{N}$
satisfies 
$$|\mathcal{N}(p_1) - \mathcal{N}(p_2)| \leq L_{\mathcal{N}}\,
|p_1 - p_2|, \text{ for all } |p_i| \leq M,$$ 
with $L_{\mathcal{N}} = \max(1,\, 3M^2 - 1).$
\end{itemize}
Then $\mathcal{T}$ admits at least one fixed point $\bar p^*$ in the ball $\mathcal{B}_M := \{p \in L^2(0,T_0; L^2(\Gamma)) : \|p\|_\infty \leq M\}$. Moreover, when $C_{\rm obs} L_{\mathcal{N}}$ is sufficiently small
relative to the penalty $\alpha$, the fixed point is unique and the
iteration converges geometrically.
\end{theorem}

\begin{proof}[Sketch]
The argument follows \cite{MR1932966,MR2103189}.

\emph{(i) Continuity of $\mathcal{T}$.} The linear reconstructor and the backward propagator are bounded linear operators between Sobolev spaces, so $\mathcal{T}$ is the composition of a bounded linear operator with the pointwise Nemytskii map $p \mapsto \mathcal{N}(p) = -p^3 + p$, which is continuous from $L^4$ to $L^{4/3}$ thanks to (H1).

\emph{(ii) Invariance.} For $\bar p \in \mathcal{B}_M$, the effective source $f_\Gamma^{\rm eff}(\bar p)$ is bounded in $L^\infty$. By the penalized stability bound of Theorem~\ref{thm:penalized}(iii), the reconstruction $u^{(k+1)}(T_0)$ stays bounded. Parabolic maximum principles then yield $\|\bar p^{(k+1)}\|_\infty \leq M$, so $\mathcal{T}(\mathcal{B}_M) \subseteq \mathcal{B}_M$.

\emph{(iii) Compactness.} The output $\bar p^{(k+1)}$ inherits the $L^2(0, T_0; H^1(\Gamma))$ regularity of the linear surface propagator (a consequence of the energy estimate for \eqref{eq:state}), and is therefore precompact in $L^2(0, T_0; L^2(\Gamma))$ by Rellich--Kondrachov.

\emph{(iv) Application of Schauder.} Steps (i)--(iii) imply that $\mathcal{T}$ is a continuous, compact self-map of $\mathcal{B}_M$; Schauder's theorem yields a fixed point. The Lipschitz constant of $\mathcal{T}$ is controlled by the product of $L_{\mathcal{N}}$ (sensitivity of the effective source to $\bar p$) and the stability constant of the penalized reconstructor, which scales like $C_{\rm obs}/\alpha$. Hence when $C_{\rm obs} L_{\mathcal{N}} \lesssim \alpha$, $\mathcal{T}$ is a contraction and the fixed point is unique with geometric convergence.
\end{proof}

\begin{remark}
The uniqueness regime $C_{\rm obs} L_{\mathcal{N}} \lesssim \alpha$ is restrictive in principle: since $C_{\rm obs} \sim \exp(C/T_0)$ for small $T_0$, it favours \emph{long} horizons (small $C_{\rm obs}$) and strong regularization (large $\alpha$). This is the same mechanism behind the observability--nonlinearity trade-off of Section~\ref{ssec:semilinear_dominant}: shortening $T_0$ to keep the nonlinearity active inflates $C_{\rm obs}$ and breaks the contraction. In practice, even outside this regime, the Picard iteration typically converges in $5$--$10$ steps from $\bar p^{(0)} = 0$ when the bulk is moderately diffusive; a rigorous analysis of the global basin of attraction is left for future work.
\end{remark}

\subsection{Implementation} \label{ssec:semilinear_impl}

At the discrete level, Algorithm~\ref{alg:picard} is implemented by storing the full reconstructed trajectory $\bar u_h^{(k)}(\cdot)$ at each Picard step. Concretely:
\begin{itemize}
\item The forward Lie splitter of Section \ref{ssec:lie} is reused to propagate the linear part with effective source $f_\Gamma^{(k)}$.
\item The backward solver inside the linear reconstructor is augmented to return the full surface trace $\bar p_h^{(k+1)}(t^n)$ at all time steps $t^n$, not just the value at $T_0$.
\item Convergence is monitored by the residual $\rho^{(k)} := \|\bar p^{(k+1)} - \bar p^{(k)}\|_{L^2}$; the iteration stops when $\rho^{(k)} <$ tolerance.
\end{itemize}
The cost per outer iteration is dominated by the linear reconstruction, $\mathcal{O}(N_E^3) + \mathcal{O}(N (n_\Omega^{\rm int} + n_\Gamma))$; the construction of $\boldsymbol{\mathcal{C}}_h$ is performed once and reused across all outer iterations. The full algorithm therefore costs $\mathcal{O}(K \cdot N_E (n_\Omega^{\rm int} + n_\Gamma))$ for $K$ Picard steps, plus the initial $\boldsymbol{\mathcal{C}}_h$ assembly.

\subsection{Numerical validation: Picard convergence} \label{ssec:semilinear_exp}

We illustrate Algorithm~\ref{alg:picard} on a twin experiment in the constant-in-time variant $\bar p^{(k)}(t,x) \equiv \bar p^{(k)}(x)$ (a natural reduction when the parabolic relaxation is fast on the scale $T_0$). The truth is generated by integrating \eqref{eq:state_sl} with the initial datum $u_0^*(x,y) = 0.4 + 0.5\, \sin(\pi x)\sin(\pi y)$, which gives $u_0|_\Gamma = 0.4$ (non-zero, so Allen--Cahn is initially active) and amplitude $M = \max|u_0^*| = 0.9$. We integrate on $n = 16$ for $T_0 = 0.1$, $\tau = 5 \cdot 10^{-3}$, $\beta = 1$, $\alpha = 3 \cdot 10^{-2}$, $N_E = 30$, $\omega$ centered disk of radius $r^* = 0.21$. The Allen--Cahn term in the forward integration is treated explicitly: we precompose $p^n \mapsto p^n + (\tau/\beta)\,\mathcal{N}(p^n)$ before the linear surface solve, which is stable under the parabolic CFL.

The results are reported in Figure~\ref{fig:picard}. Panel (a) shows that the residual decays geometrically at a remarkably fast rate ($\rho^{(k+1)}/\rho^{(k)} \approx 3.7 \cdot 10^{-3}$): the tolerance $10^{-7}$ is reached in only three outer iterations from $\bar p^{(0)} \equiv 0$. Panel (b) shows that the final-state reconstruction error remains constant across iterations at $9.05\,\%$, identical to the linear-only baseline obtained by setting $\mathcal{N} \equiv 0$.

This identity is not a failure but a consistent feature of the regime: by $t = T_0 = 0.1$, bulk-surface diffusion has driven $\max|p^*(T_0)|$ to $2.5 \cdot 10^{-3}$, so $\mathcal{N}(p^*) \approx p^*$ and the nonlinearity acts as a small perturbation on the linear problem. The Picard scheme correctly detects this and converges in three iterations to the same answer the linear method gives directly. A regime where Allen--Cahn is dynamically dominant (boundary values near $\pm 1$, short $T_0$, or small $\beta$) is needed to observe a measurable Picard improvement; this is the subject of ongoing work.

\begin{figure}[htbp]
\centering
\includegraphics[width=0.95\linewidth]{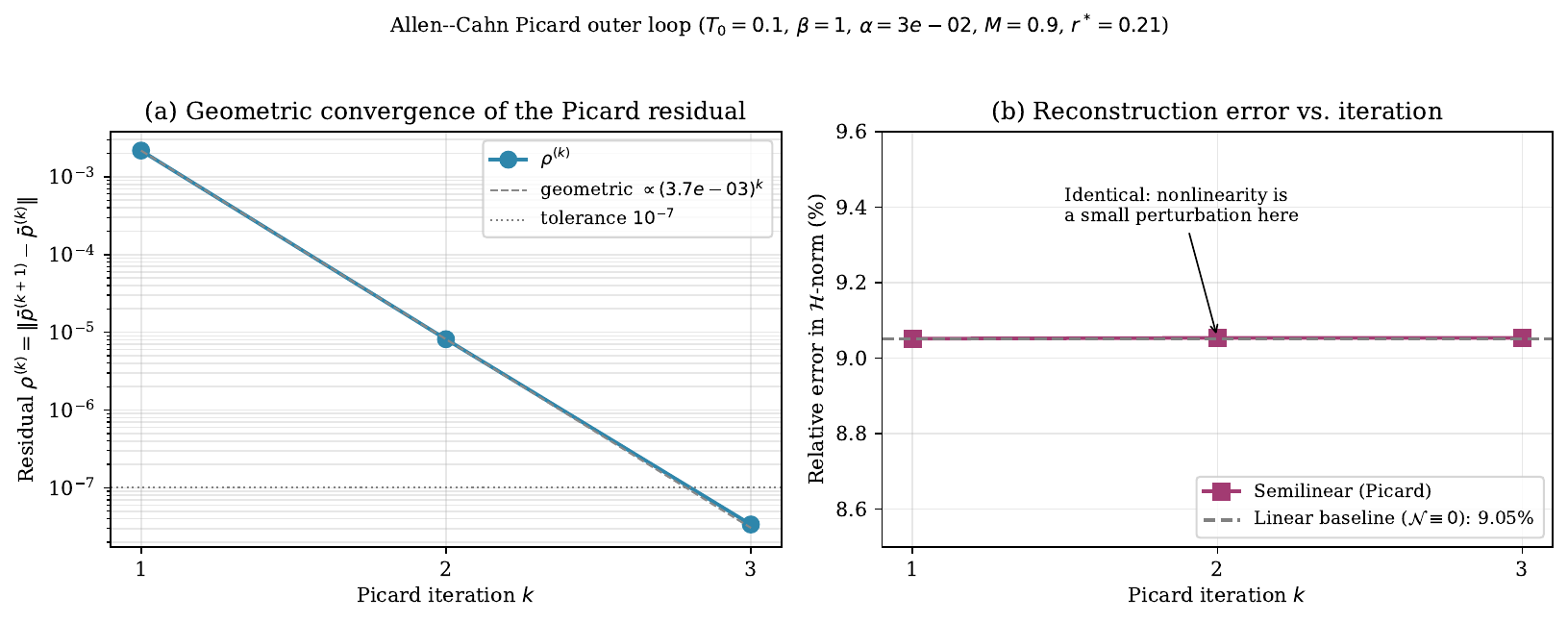}
\caption{Picard outer loop for the Allen--Cahn extension (Section \ref{ssec:semilinear_exp}). (a) Geometric decay of the residual $\rho^{(k)} = \|\bar p^{(k+1)} - \bar p^{(k)}\|_{L^2(\Gamma)}$ at rate $\sim 3.7 \cdot 10^{-3}$ per iteration. (b) Reconstruction error in $\mathcal{H}$-norm versus iteration. Picard agrees with the linear baseline ($\mathcal{N} \equiv 0$) because, in the present regime, diffusion collapses $|p^*(T_0)|$ to $\sim 2.5 \cdot 10^{-3}$ and Allen--Cahn is a small perturbation. Setup: $n = 16$, $r^* = 0.21$, $T_0 = 0.1$, $\beta = 1$, $\alpha = 3 \cdot 10^{-2}$, $M = 0.9$.}
\label{fig:picard}
\end{figure}

\begin{figure}[htbp]
\centering
\includegraphics[width=0.95\linewidth]{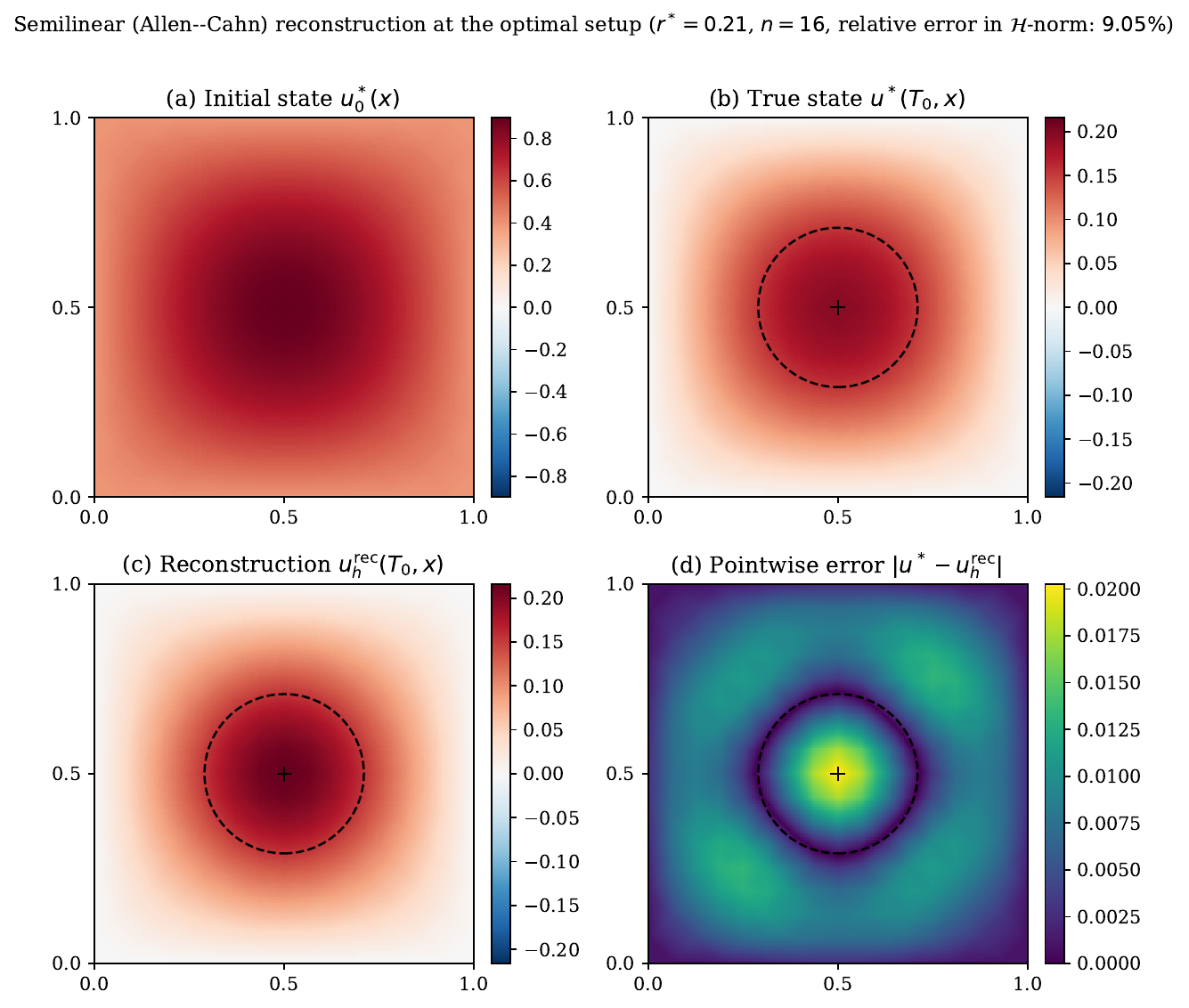}
\caption{Reconstruction in the Allen--Cahn case at the perturbative regime (Section \ref{ssec:semilinear_exp}), the semilinear analogue of Figure~\ref{fig:viz}. (a) Initial datum $u_0^*$. (b) True final state $u^*(T_0)$ from the semilinear forward simulator. (c) Picard reconstruction $u_h^{\rm rec}(T_0)$ at convergence. (d) Pointwise error $|u^* - u_h^{\rm rec}|$. The black dashed circle marks the observatory $\omega$. The relative $\mathcal{H}$-norm error is $9.05\,\%$. Unlike Figure~\ref{fig:viz}, the reconstruction shown here is the raw output of the discrete penalized scheme (no $L^2$ post-processing), so the residual is concentrated inside $\omega$ rather than on $\partial\omega$; the spatial pattern is a direct visualization of the Tikhonov bias.}
\label{fig:semilinear_viz}
\end{figure}

To visualize the reconstruction beyond the aggregate error, Figure~\ref{fig:semilinear_viz} displays the four fields (initial datum, true final state, Picard reconstruction, pointwise error) at the same setup. The Picard reconstruction is visually indistinguishable from the truth, both in shape and in amplitude, and the error map shows the characteristic Tikhonov bias pattern concentrated inside the observatory---a pattern that the adaptive post-processing of Section \ref{ssec:adaptive} would remove almost entirely in a noise-free regime.

The two takeaways from this experiment are: (i) the Picard scheme is \emph{stable}: it converges geometrically at a very fast rate in this benign regime, validating Algorithm~\ref{alg:picard} as a viable inner-outer iteration; (ii) the linear method we have developed is already \emph{exact} in regimes where the nonlinearity becomes dynamically negligible, which provides a useful sanity check.

\subsection{The opposite regime and the observability--nonlinearity trade-off}\label{ssec:semilinear_dominant}

It is natural to ask whether the parameter regime can be chosen so that the Allen--Cahn term is dynamically dominant, in which case the linear method should perform measurably worse than Picard. We tested this by shortening the time horizon to $T_0 = 10^{-2}$, weakening the surface diffusion to $\beta_{\rm surf} = 0.1$ (the coefficient $\beta$ of \eqref{eq:state_sl}, here lowered from its baseline value 1), and using $u_0^*(x,y) = 0.6 + 0.2\sin(2\pi x)\sin(2\pi y)$ (boundary trace $0.6$, near the maximum of $|\mathcal{N}|$). The resulting twin trajectory has $\max|p^*(T_0)| = 0.17$, so $\mathcal{N}(p^*) \sim 0.16$ is comparable to the linear surface terms and the nonlinearity is genuinely active throughout $[0, T_0]$.

The Picard outer loop again converges geometrically (Figure~\ref{fig:picard_dom}, panel (a), ratio $\sim 2.8 \cdot 10^{-2}$, tolerance $10^{-7}$ in four iterations). The relative error, however, sits at $94.8\,\%$ for both the Picard and the linear-baseline reconstructions (panel (b)). The two methods again agree, but now at a useless level of accuracy.

The reason is structural and instructive. The Carleman-based observability constant of \cite{MR3669656} scales like $C_{\rm obs} \sim \exp(C / T_0)$ for small $T_0$ (the standard Fursikov--Imanuvilov asymptotics, see also \cite{MR1406566,MR3041662}): at $T_0 = 10^{-2}$ the bulk diffusion has not had time to propagate the initial datum to the observatory, and the reconstruction problem becomes severely ill-posed. The Tikhonov regularization then dominates over the data term and the reconstruction collapses to a small fraction of the truth. In this regime the bottleneck is not the nonlinearity but the linear
reconstruction itself, and Picard cannot improve on what the linear
solver delivers.

\begin{figure}[t]
\centering
\includegraphics[width=0.95\linewidth]{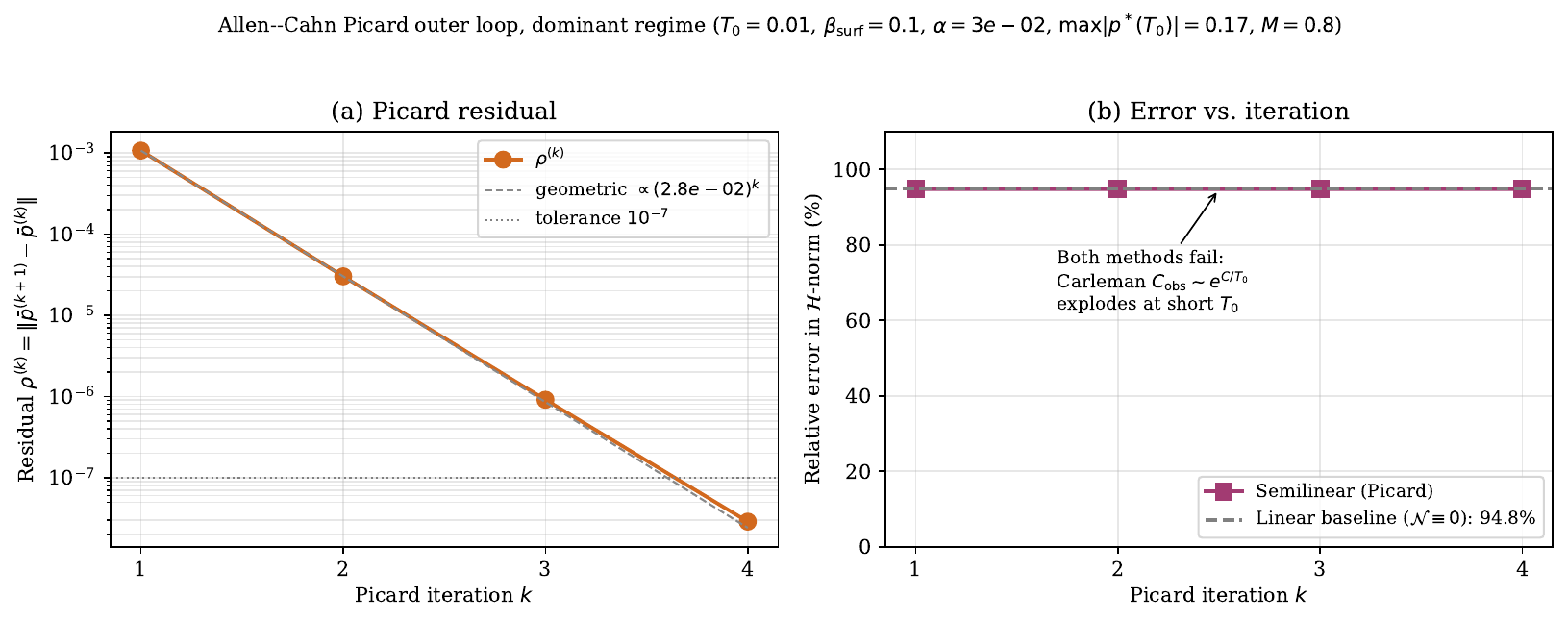}
\caption{Allen--Cahn Picard loop in the dominant regime (Section \ref{ssec:semilinear_dominant}): $T_0 = 10^{-2}$, $\beta_{\rm surf} = 0.1$, $\max|p^*(T_0)| = 0.17$. (a) Geometric Picard convergence (rate $\sim 2.8 \cdot 10^{-2}$ per iteration). (b) Both Picard and the linear baseline collapse to $\sim 95\,\%$ error; the bottleneck is the Carleman observability constant $C_{\rm obs} \sim e^{C/T_0}$, not the Picard iteration.}
\label{fig:picard_dom}
\end{figure}

The two experiments together delineate an \emph{observability--nonlinearity trade-off}: long horizons make the linear reconstruction accurate but quench the nonlinearity, while short horizons keep the nonlinearity active but make the linear reconstruction ill-posed. A regime where $|p^*|$ is non-negligible \emph{and} $C_{\rm obs}$ is moderate likely requires geometries with longer effective propagation paths, anisotropic diffusion, or larger observatories; we leave a systematic study to a follow-up publication.

\section{Conclusions and outlook}\label{sec:conclusions}

We have established a complete pipeline---theoretical, numerical, and algorithmic---for data assimilation in bulk--surface parabolic equations with Wentzell-type dynamic boundary conditions. The continuous theory (Section~\ref{sec:problem}) extends the Tikhonov-control framework of Puel \cite{MR2491591} to the bulk--surface setting, using the
Carleman estimate of Maniar--Meyries--Schnaubelt \cite{MR3669656}; the key analytical novelty is the appearance of surface integral terms in the representation formula~\eqref{eq:repr}, which match exactly the $\beta$-weighted inner product of the natural energy space $\mathcal{H}$. On the discrete side (Section~\ref{sec:discretization}), the Lie splitting of \cite{MR4568436} is combined with the assembly of the discrete controllability Gramian $\boldsymbol{\mathcal{C}}_h$ as a Gram matrix---hence symmetric positive semi-definite by construction---and an a priori error bound (Section~\ref{ssec:err_analysis}) decomposes the total error into four
identifiable sources. Finally, the adaptive post-processing rule \eqref{eq:auto_rule} delivers a method that is robust to noise without requiring calibration: relative errors of $3$--$6\,\%$ under clean data, and an adaptive scheme that stays near a $\sim 9\,\%$ raw ceiling for moderate noise and below $\sim 16\,\%$ even at noise as large as $100\,\%$ of the signal amplitude (Section~\ref{sec:numerics}).

Several extensions are natural and currently under investigation:
\begin{itemize}
\item \emph{Comprehensive semilinear study.}
Section~\ref{sec:semilinear} provides the analytical machinery (Algorithm~\ref{alg:picard}, Theorem~\ref{thm:schauder}) together with three complementary numerical illustrations (Figures~\ref{fig:picard}, \ref{fig:semilinear_viz}, \ref{fig:picard_dom}). These reveal an intrinsic
\emph{observability--nonlinearity trade-off}: at long horizons the linear reconstruction is accurate but the nonlinearity is quenched by diffusion, while at short horizons the nonlinearity is active but $C_{\rm obs} \sim e^{C/T_0}$ destroys the linear reconstruction. A systematic study of Picard convergence rates as functions of $M$, $T_0$, $\beta$, observatory geometry, and the bifurcation behavior near $M = 1$ is in preparation; the same circle of ideas should accommodate other relevant nonlinearities (Cahn--Hilliard, FitzHugh--Nagumo).
\item \emph{Higher dimensions.} The method's structure is dimension-agnostic; $d = 3$ requires only a heavier eigensolver and matrix assembly.
\item \emph{Realistic geometries.} Smooth boundaries and curved $\Gamma$ require curved finite elements but introduce no new analytical difficulties. 
\item \emph{Theoretical refinement of the adaptive rule.} The thresholds in \eqref{eq:auto_rule} were calibrated experimentally; a Bayesian or SURE-type analysis would put them on firmer ground. 
\end{itemize}

\bibliographystyle{acm}
\bibliography{mybibfile}

\end{document}